\documentclass[11pt]{amsart}

\usepackage{epsfig}
\usepackage{amsmath, amsthm, amssymb}

\newtheorem{defn}{Definition}[section]

\newtheorem{cor}[defn]{Corollary}
\newtheorem{thm}[defn]{Theorem}
\newtheorem{prop}[defn]{Proposition}
\newtheorem{lemma}[defn]{Lemma}

\newtheorem{remark}[defn]{Remark}

\newtheorem{cond}[defn]{Condition}

\newcommand{\be}{\begin{equation}}
\newcommand{\ee}{\end{equation}}
\newcommand{\bea}{\begin{eqnarray}}
\newcommand{\eea}{\end{eqnarray}}
\newcommand{\beas}{\begin{eqnarray*}}

\newcommand{\eeas}{\end{eqnarray*}}

\newcommand{\R}{\mathbb{R}}

\newcommand{\ve}{\varepsilon}

\newcommand{\noi}{\noindent}

\newcommand{\goto}{\rightarrow}
\newcommand{\hsp}{\hspace{.3in}}
\newcommand{\bp}{\begin{proof}}
\newcommand{\ep}{\end{proof}}

\newcommand{\TV}{\tiny{\mbox{TV}}}
\newcommand{\mix}{\tiny{\mbox{mix}}}

\newcommand{\dstyle}{\displaystyle}

\begin{document}

\title[Path Coupling and  Aggregate Path Coupling]{Path Coupling and  Aggregate Path Coupling}

\author{Yevgeniy Kovchegov}
\address{Department of Mathematics, Oregon State University, Corvallis, OR  97331}
\email{kovchegy@math.oregonstate.edu}
\thanks{This work was partially supported by a grant from the Simons Foundation (\#284262 to Yevgeniy Kovchegov).}

\author{Peter T. Otto}
\address{Department of Mathematics, Willamette University, Salem, OR 97302,  phone: 1-503-370-6487}
\email{potto@willamette.edu}


\begin{abstract}
In this survey paper, we describe and characterize an extension to the classical path coupling method applied statistical mechanical models, referred to as aggregate path coupling.  In conjunction with large deviations estimates, we use this aggregate path coupling method to prove rapid mixing of Glauber dynamics for a large class of statistical mechanical models, including models that exhibit discontinuous phase transitions which have traditionally been more difficult to analyze rigorously.  The parameter region for rapid mixing for the generalized Curie-Weiss-Potts model is derived as a new application of the aggregate path coupling method.
\end{abstract}

\date{\today}

\maketitle

\tableofcontents

\section{Introduction}  

The theory of mixing times addresses a fundamental question that lies at the heart of statistical mechanics. How quickly does a physical system relax to equilibrium? A related problem arises in computational statistical physics concerning the accuracy of computer simulations of equilibrium data. One typically carries out such simulations by running Glauber dynamics or the closely related Metropolis algorithm, in which case the theory of mixing times allows one to quantify the running time required by the simulation.

\medskip

An important question driving the work in the field is the relationship between the mixing times of the dynamics and the equilibrium phase transition structure of the corresponding statistical mechanical models.  Many results for models that exhibit a continuous phase transition were obtained by a direct application of the standard path coupling method.  Standard path coupling \cite{BD} is a powerful tool in the theory of mixing times of Markov chains in which rapid mixing can be proved by showing that the mean coupling distance contracts between all neighboring configurations of a minimal path connecting two arbitrary configurations.

\medskip

For models that exhibit a discontinuous phase transition, the standard path coupling method fails.  In this survey paper, we show how to combine aggregate path coupling and large deviation theory to determine the mixing times of a large class of statistical mechanical models, including those that exhibit a discontinuous phase transition.  The aggregate path coupling method extends the use of the path coupling technique in the absence of contraction of the mean coupling distance between all neighboring configurations of a statistical mechanical model.  The primary objective of this survey is to characterize the assumptions required to apply this new method of aggregate path coupling.

\medskip

The manuscript is organized as follows: in Section \ref{mixtimepathcoupling}, we give a brief overview of mixing times, coupling and path coupling methods, illustrated with a new example of path coupling. Then, beginning in Section \ref{gibbs}, we introduce the class of statistical mechanical models considered in the survey.  In Sections \ref{MFBC} and \ref{genspinmodels}, we develop and characterize the theory of aggregate path coupling and apply it in Section  \ref{GCWP}, where we derive the parameter region for rapid mixing for the generalized Curie-Weiss-Potts model that was introduced recently in \cite{JKRW}.

\bigskip

\section{Mixing Times and Path Coupling} \label{mixtimepathcoupling} 

The mixing time is a measure of the convergence rate of a Markov chain to its stationary distribution and is defined in terms of 
the {\it total variation distance} between two distributions $\mu$ and $\nu$ defined by
\[ \| \mu - \nu \|_{\TV} = \sup_{A \subset \Omega} | \mu(A) - \nu(A)| = \frac{1}{2} \sum_{x \in \Omega} | \mu(x) - \nu(x)| \]
Given the convergence of the Markov chain, we define the {\it maximal distance to stationary} to be 
\[ d(t) = \max_{x \in \Omega} \| P^t(x, \cdot) - \pi \|_{\TV} \]
where $P^t(x, \cdot)$ is the transition probability of the Markov chain starting in configuration $x$ and $\pi$ is its stationary distribution.  Rather than obtaining bounds on $d(t)$, it is sometimes easier to bound the standardized maximal distance defined by 
\be
\label{eqn:StanMaxDist}
\bar{d}(t) := \max_{x,y \in \Omega} \|P^t(x, \cdot) - P^t(y, \cdot) \|_{\TV} 
\ee
which satisfies the following result.

\begin{lemma} {\em (\cite{LPW} Lemma 4.11)} \ With $d(t)$ and $\bar{d}(t)$ defined above, we have
\[ d(t) \leq \bar{d}(t) \leq 2 \, d(t). \]
\end{lemma}

\medskip

Given $\ve>0$, the {\it mixing time} of the Markov chain is defined by
\[ t_{\mix}(\ve) = \min\{ t : d(t) \leq \ve \} \smallskip \]
In the modern theory of Markov chains, the interest is in the mixing time as a function of the system size $n$ and thus, for emphasis, we will often use the notation $t_{\mix}^{(n)}(\ve)$. 
With only a handful of general techniques, rigorous analysis of mixing times is difficult and the proof of exact mixing time asymptotics (with respect to $n$) of even some basic chains remain elusive. See \cite{LPW} for a survey on the theory of mixing times. 

\medskip

Rates of mixing times are generally categorized into two groups: {\it rapid mixing} which implies that the mixing time exhibits polynomial growth with respect to the system size, and {\it slow mixing} which implies that the mixing time grows exponentially with the system size.  Determining where a model undergoes rapid mixing is of major importance, as it is in this region that the application of the dynamics is physically feasible.  

\bigskip

\subsection{Coupling Method}  

The application of coupling (and path coupling) to mixing times of Markov chains begins with the following lemma:

\begin{lemma}
Let $\mu$ and $\nu$ be two probability distributions on $\Omega$.  Then 
\[ \| \mu - \nu \|_{\text{{\em \TV}}} = \inf \{ P\{ X \neq Y \} : (X,Y) \mbox{ is a coupling of $\mu$ and $\nu$} \} \]
\end{lemma}
This lemma implies that the total variation distance to stationarity, and thus the mixing time, of a Markov chain can be bounded above by the probability $P(X_t \not= Y_t)$ for a coupling of the Markov chain $(X_t, Y_t)$ starting in different configurations; i.e. $(X_0,Y_0)=(\sigma, \tau)$, or if one of the coupled chains is  distributed by the stationary distribution $\pi$ for all $t$.  

\medskip

We run the coupling of the Markov chain, not necessarily independently, until they meet at time $\tau_c$. This is called the {\it coupling time}. After $\tau_c$, we run the chains together. We see that $X_t$ must have the stationary distribution for  $t \geq \tau_c$, since $X_t  = Y_t$ after coupling.

%
%

\begin{thm}[The Coupling Inequality]
Let $(X_t, Y_t)$ be a coupling of a Markov chain where $Y_t$ is distributed by the stationary distribution $\pi$.  The {\bf coupling time} of the Markov chain is defined by 
\[ \tau_{c} := \min\{t: X_t = Y_t\}. \]
Then, for all initial states $x$, 
\[ \|P^t(x,\cdot) - \pi \|_{\mbox{{\em \TV}}}  \leq P(X_t \neq Y_t) = P \ [\tau_{c} > t] \]
and thus $\tau_{\mbox{{\em \mix}}}(\ve) \leq \mathbb E[ \tau_{c} / \ve]$.
\end{thm}
From the Coupling Inequality, it is clear that in order to use the coupling method to bound the mixing time of a Markov chain, one needs to bound the coupling time for a coupling of the Markov chain starting in {\it all} pairs of initial states.  The advantage of the {\it path coupling method} described in the next section is that it only requires a bound on couplings starting in certain pairs of initial states.

\bigskip

\subsection{Path Coupling} \label{pathcoupling} 

The idea of the path coupling method is to view a coupling that starts in configurations $\sigma$ and $\tau$ as a sequence of couplings that start in {\it neighboring} configurations $(x_i, x_{i+1})$ such that $( \sigma = x_0, x_1, x_2, \ldots, x_r = \tau)$.  Then the contraction of the original coupling distance can be obtained by proving contraction between neighboring configurations which is often easier to show. 

\medskip
\noi
Let $\Omega$ be a finite sample space, and suppose $(X_t, Y_t)$ is a coupling of a Markov chain on $\Omega$.  Suppose also there is a neighborhood structure on $\Omega$, and suppose it is transitive in the following sense: for any $x$ and $y$, there is a neighbor-to-neighbor path $$x \sim x_1 \sim x_2 \sim \hdots \sim x_{r-1} \sim y,$$ 
where $u \sim v$ denotes that sites $u$ and $v$ are neighbors.\\

\noi
Let $d(x,y)$ be a metric over $\Omega$ such that $d(x,y) \geq 1$ for any $x \not= y$, and 
$$d(x,y)=\min\limits_{\rho: x \rightarrow y} \sum_{i=1}^r d(x_{i-1},x_i),$$
where the minimum is taken over all neighbor-to-neighbor paths 
$$\rho: ~x_0=x ~\sim x_1 \sim x_2 \sim \hdots \sim x_{r-1} \sim ~x_r=y$$
of any number of steps $r$. Such metric is called {\it path metric}. Next, we define the {\it diameter} of the sample space:
$$\text{diam} (\Omega)=\max\limits_{x,y \in \Omega} d(x,y).$$

\noi
Finally, the coupling construction allows us to define the {\it transportation metric} of Kantorovich (1942) as follows:
$$d_K(x,y):= E[d(X_{t+1},Y_{t+1}) ~|X_t=x,Y_t=y].$$
One can check $d_K(x,y)$ is a metric over $\Omega$.\\

\noi
Path coupling, invented by Bubley and Dyer in 1997,  is a method that employs an existing coupling construction in order to bound the mixing time from above. This method in its standard form usually requires certain metric contraction between neighbor sites.  Specifically,  we require that for any $x \sim y$,
\be \label{eqn:contr1}
d_K(x,y)=E[d(X_{t+1},Y_{t+1}) ~|X_t=x,Y_t=y] ~\leq ~\big(1-\delta(\Omega)\big) d(x,y),
\ee
where $0<\delta(\Omega)<1$ does not depend on $x$ and $y$.\\

\noi
The above contraction inequality (\ref{eqn:contr1}) has the following implication. 

\begin{prop}\label{classicalPC}
Suppose inequality {\em (\ref{eqn:contr1})} is satisfied. Then
$$t_{\text{{\em \mix}}}(\epsilon) \leq \left\lceil  {\log\text{{\em diam}} (\Omega)-\log\epsilon \over \delta(\Omega)} \right\rceil .$$
\end{prop}
\begin{proof}
For any $x$ and $y$ in $\Omega$, consider the path metric minimizing path 
$$\rho: ~x_0=x ~\sim x_1 \sim x_2 \sim \hdots \sim x_{r-1} \sim ~x_r=y$$ 
such that
$$d(x,y)=\sum_{i=1}^r d(x_{i-1},x_i).$$
Then
\beas
E[d(X_{t+1},Y_{t+1}) ~|X_t=x,Y_t=y] & = & d_K(x,y)\\
& \leq & \sum_{i=1}^r d_K(x_{i-1},x_i) \\
& \leq & (1-\delta(\Omega)\big)  \sum_{i=1}^r d(x_{i-1},x_i) \\
& = & (1-\delta(\Omega)\big) d(x,y).
\eeas
Hence, after $t$ iterations,
$$E[d(X_t,Y_t)] \leq (1-\delta(\Omega)\big)^t d(X_0,Y_0) \leq \big(1-\delta(\Omega)\big)^t \text{diam} (\Omega)$$
for any initial $(X_0,Y_0)$, and
$$P(X_t \not= Y_t) = P\big( d(X_t,Y_t) \geq 1 \big) \leq E[d(X_t,Y_t)] \leq \big(1-\delta(\Omega)\big)^t \text{diam} (\Omega) \leq \epsilon$$
whenever 
$$t \geq {\log\text{diam} (\Omega)-\log\epsilon \over -\log\big(1-\delta(\Omega)\big)}.$$
Thus, by the Coupling Inequality,
\[ t_{\mix}(\epsilon) \leq \left\lceil  {\log\text{diam} (\Omega)-\log\epsilon \over -\log\big(1-\delta(\Omega)\big)} \right\rceil 
\leq \left\lceil  {\log\text{diam} (\Omega)-\log\epsilon \over \delta(\Omega)} \right\rceil . \]
\end{proof}

\noi
{\bf Example.} Consider the Ising model over a $d$-dimensional torus $\mathbb{Z}^d /n\mathbb{Z}^d$. There $\Omega=\{-1,+1\}^{n^d}$ is the space of all spin configurations, and for any pair of configurations $\sigma,\tau \in \Omega$, the path metric $d(\sigma,\tau)$ is the number of discrepancies between them
$$d(\sigma, \tau) = \sum_{x \in \mathbb{Z}^d /n\mathbb{Z}^d} \mathbf{1}\{ \sigma_x \neq \tau_x \}.$$
The diameter $~\text{diam} (\Omega)=n^d$. It can be checked that if the inverse temperature parameter $\beta$ satisfies $\tanh(\beta)<{1 \over 2d}$, the contraction inequality (\ref{eqn:contr1}) is satisfied with 
$$\delta(n)={1-2d\tanh(\beta) \over n}.$$
Hence
$$t_{mix}(\epsilon) \leq \left\lceil  {\log\text{diam} (\Omega)-\log\epsilon \over \delta(\Omega)} \right\rceil
=\left\lceil  n{d\log n-\log\epsilon \over 1-2d\tanh(\beta)} \right\rceil =O\Big(Cn\log n \Big),$$ 
where $~C={d \over 1-2d\tanh(\beta)}$.

\vskip 0.3 in

The emergence of the path coupling technique \cite{BD} has allowed for a greater simplification in the use of the coupling argument, as rigorous analysis of coupling can be significantly easier when one considers only neighboring configurations. However, the simplification of the path coupling technique comes at the cost of the strong assumption that the coupling distance for all pairs of neighboring configurations must be contracting. Observe that although the contraction between all neighbors is a sufficient condition for the above mixing time bound, it is far from being a necessary condition. In fact, this condition is an artifact of the method.

\medskip

There had been some successful generalizations of the path coupling method. Specifically in \cite{DGJM}, \cite{HV} and \cite{BordD}. In \cite{DGJM}, the path coupling method is generalized to account for contraction after a specific number of time-steps, defined as a random variable.  In \cite{HV} a multi-step non-Markovian coupling construction is considered that evolves via partial couplings of variable lengths determined by stopping times.  In order to bound the coupling time, the authors of  \cite{HV} introduce a technique they call {\it variable length path coupling} that further generalizes the approach in \cite{DGJM}. 

\bigskip

\subsection{Random-to-Random Shuffling} 

An example illustrating the idea of path coupling can be found in the REU project \cite{reu} of Jennifer Thompson that was supervised by Yevgeniy Kovchegov in the summer of 2010 at Oregon State University. There, we consider the shuffling algorithm whereby on each iteration we select a card uniformly from the deck, remove it from  the deck, and place it in one of the $n$ positions in the deck, selected uniformly and independently. Each iteration being done independently of the others. This is referred to as the random-to-random card shuffling algorithm.  We need to shuffle the deck so that when we are done with shuffling the deck each of $n!$ possible permutations is obtained with probability close to ${1 \over n!}$. Its mixing time can be easily shown to be of order $O(n \log{n})$ using the notion of {\it strong stationary time}. For this one would consider the time it takes for each card in the deck to be selected at least once. Then use the {\it coupon collector problem} to prove the $O(n \log{n})$ upper bound on the mixing time. The same coupon collector problem is applied to show that we need at least $O(n \log n)$ iterations of the shuffling algorithm to mix the deck. The goal of the REU project in \cite{reu} was to arrive with the $O(n \log{n})$ upper bound using the coupling method. 

\subsubsection{The Coupling}

Take two decks of $n$ cards,  A and B.

\begin{itemize}
\item Randomly choose $i \in [1,n]$.
\item Remove card with label $i$ from each deck.
\item Randomly reinsert card $i$ in deck A.
\item \begin{enumerate}
\item If the new location of $i$ in A is the top of A, then insert $i$ on the top of B.
\item If the new location of $i$ in A is below card $j$, insert $i$ below $j$ in B.
\end{enumerate}
\end{itemize}

\begin{figure}[ht]
\begin{center} 
 \includegraphics[scale=1.0]{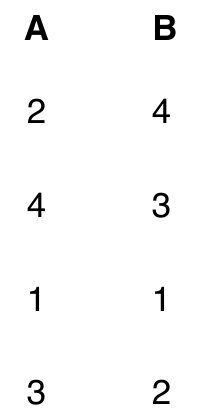} 
\end{center}
\caption{One configuration of matchings between two decks of $n=4$ cards.}
\label{perm0}
\end{figure}

Let $A_t \in S_n$ and $B_t \in S_n$ denote the card orderings (permutations)  in decks A and B after $t$ iterations.

\subsubsection{Computing the coupling time with a laces approach}

We introduce the following path metric $d(\cdot,\cdot)  : ~S_n \times S_n \rightarrow \mathbb Z_+$ by letting  $d(\sigma, \sigma ')$ be the minimal number of nearest neighbour transpositions to traverse between the two permutations, $\sigma$ and $\sigma'$. For example, for the two decks A and B in Figure \ref{perm0}, a distance minimizing path connecting the two permutations is given in Figure \ref{perm1}.
\begin{figure}[ht]
\begin{center} 
  \includegraphics[scale=1.0]{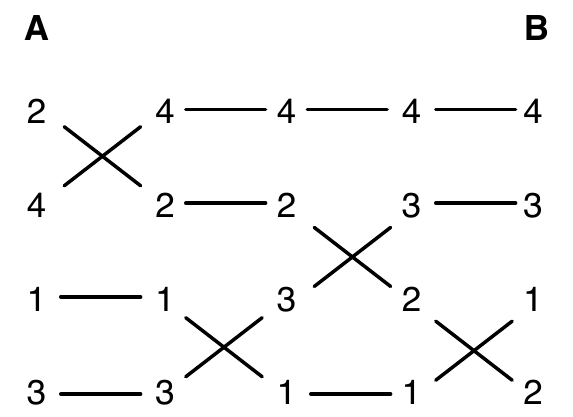} 
\end{center}
\caption{Minimal number of crossings between the two permutations is four.} 
\label{perm1}
\end{figure}

\medskip
\noindent
Note that $d(\sigma, \sigma ') \leq  { n \choose 2}$. We consider the quantity $d_t = d(A_t, B_t)$, the distance between our two decks at time $t$. We want to find the relationship between $\mathbb E[d_{t+1}]$ and $\mathbb E[d_{t}]$.

\medskip
\noindent
We consider a $d(\cdot,\cdot)$-metric minimizing path. We call the path taken by a card label a {\it lace}. Thus each lace representing a card label is involved in a certain number of crossings. Let $r_t$  be the number crossings per lace, averaged over all $n$ card labels. Then we have $d_t = \frac{nr_t}{2}$.

\begin{figure}[ht]
\begin{center} 
  \includegraphics[scale=1.0]{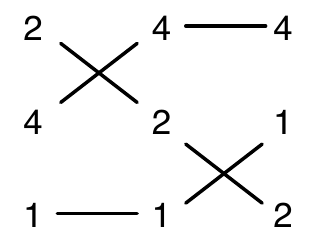} 
\end{center}
\caption{Removing lace $3$ decreases the number of crossings to two.} 
\label{perm2}
\end{figure}

\medskip
\noindent
The evolution of the path connecting $A_t$ to $B_t$  can be described as following. At each timestep we pick a lace (corresponding to a card label, say $i$) at random and remove it. For example, take a minimal path connecting decks A and B in Figure \ref{perm1}, and remove a lace corresponding to label $3$, obtaining Figure \ref{perm2}.
\medskip
\noindent
Then we reinsert the removed lace back. There will be two cases:
\begin{enumerate}
\item With probability $\frac {1}{n}$ we place the lace corresponding to card label $i$ to the top of the deck. See Figure \ref{perm3}. Then there will be no new crossings.

\begin{figure}[ht]
\begin{center} 
  \includegraphics[scale=1.0]{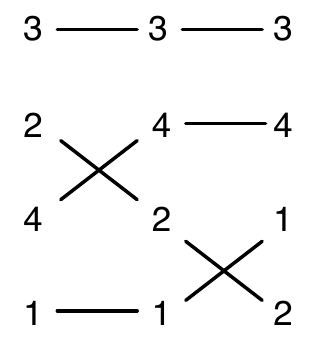} 
\end{center}
\caption{Placing lace $3$ on top does not add new crossings.} 
\label{perm3}
\end{figure}
 
\item We choose a lace $j$ randomly and uniformly chosen among the remaining $n-1$ laces, and place lace $i$ directly below lace $j$. This has probability $\frac{n-1}{n}$. Then the number of additional new crossings is the same as the number of crossings of lace $j$, as in Figure \ref{perm4}.
Here $$ \mathbb E [\textrm{new crossings}] = \mathbb E [\textrm{average number of crossings for the remaining laces} ] = \left(\frac{nr_t} 2 - r_t \right) \frac {1}{n-1}.$$
\end{enumerate}

\begin{figure}[ht]
\begin{center} 
  \includegraphics[scale=1.0]{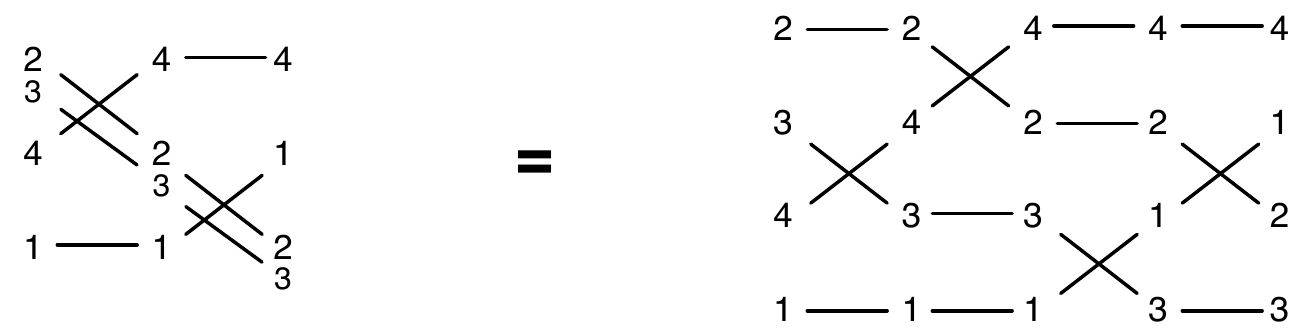} 
\end{center}
\caption{Inserting lace $3$ directly below lace $2$ adds the same number of crossing as there were of lace $2$.} 
\label{perm4}
\end{figure}

\noindent
Then 
$$\mathbb E[d_{t+1} | A_t, B_t] = \frac{nr_t} 2 - r_t + \left(\frac{n-1}{n} \right) \left( \frac{nr_t} 2 - r_t\right)\frac {1}{n-1}= \left( 1 - \frac{1}{n} - \frac{2}{n^2} \right) d_t .$$

\noindent
Hence
$$\mathbb E[d_{t+1}] = \left( 1 - \frac{1}{n} - \frac{2}{n^2} \right) \mathbb E[d_t],$$
and therefore
$$P(A_t \not= B_t)=P(d_t \geq 1) \leq  \mathbb E[d_t] = \left(1 - \frac{1}{n} - \frac{2}{n^2}\right)^t \mathbb E[d_0] \leq \left( 1 - \frac{1}{n} - \frac{2}{n^2}\right)^t {n \choose 2} \leq \epsilon$$
whenever
$$ t \geq {- 2 \log n + \log 2 +\log \epsilon \over \log  \left( 1 - \frac 1 n - \frac 2 {n^2} \right)} = 2 n \log n + O (n).$$

\medskip
\noindent
Thus providing an upper bound on mixing time.

\bigskip

\section{Gibbs Ensembles and Glauber Dynamics} \label{gibbs} 

\medskip

In recent years, mixing times of dynamics of statistical mechanical models have been the focus of much probability research, drawing interest from researchers in mathematics, physics and computer science.  The topic is both physically relevant and mathematically rich.  But up to now, most of the attention has focused on particular models including rigorous results for several mean-field models.  A few examples are (a) the Curie-Weiss (mean-field Ising) model \cite{DLP1, DLP2, EL, LLP}, (b) the mean-field Blume-Capel model \cite{EKLV, KOT}, (c) the Curie-Weiss-Potts (mean-field Potts) model \cite{BR, CDLLPS}.  A good survey of the topic of mixing times of statistical mechanical models can be found in the recent paper by Cuff et.\ al.\ \cite{CDLLPS}.

The aggregate path coupling method was developed in \cite{KOT} and \cite{KO} to obtain rapid mixing results for statistical mechanical models, in particular, those models that undergo a first-order, discontinuous phase transition.  For this class of models, the standard path coupling method fails to be applicable.  The remainder of this survey is devoted to the exposition of the aggregate path coupling method applied to statistical mechanical models.

\medskip

As stated in \cite{Ellis}, ``In statistical mechanics, one derives macroscopic properties of a substance from a probability distribution that describes the complicated interactions among the individual constituent particles.''  The distribution referred to in this quote is called the Gibbs ensemble or Gibbs measure which are defined next.

\vskip 0.1 in
\noindent

A {\it configuration} of the model has the form $\omega = (\omega_1, \omega_2, \ldots, \omega_n) \in \Lambda^n$, where $\Lambda$ is some finite, discrete set.  We will consider a configuration on a graph with $n$ vertices and let $X_i(\omega) = \omega_i$ denote the {\it spin} at vertex $i$.  The random variables $X_i$'s for $i=1, 2, \ldots, n$ are independent and identically distributed with common distribution $\rho$.  The interactions among the spins are defined through the {\it Hamiltonian} function $H_n$ and we denote by $M_n(\omega)$ the relevant macroscopic quantity corresponding to the configuration $\omega$.  The lift from the microscopic level of the configurations to the macroscopic level of $M_n$ is through the {\it interaction representation function} $H$ that satisfies 
\be
\label{eqn:IntRepFunct} 
H_n(\omega) = n H(M_n(\omega)). 
\ee

\begin{defn}
\label{defn: genGibbs}
The {\bf Gibbs measure} or {\bf Gibbs ensemble} in statistical mechanics is defined as
\be 
\label{eqn:gibbs}
P_{n, \beta} \, (B) = \frac{1}{Z_n(\beta)} \int_B \exp \left\{ -\beta H_n(\omega) \right\} dP_n = \frac{1}{Z_n(\beta)} \int_B \exp \left\{  -\beta n \, H\left(M_n(\omega) \right) \right\} dP_n 
\ee
where $P_n$ is the product measure with identical marginals $\rho$ and $Z_n(\beta) = \int_{\Lambda^n} \exp \left\{ -\beta H_n(\omega) \right\} dP_n$ is the {\bf partition function}.  The positive parameter $\beta$ represents the inverse temperature of the external heat bath.
\end{defn}

\begin{defn}
\label{defn: genGlauber}
On the configuration space $\Lambda^n$, we define the {\bf Glauber dynamics} for the class of spin models considered in this paper.  These dynamics yield a reversible Markov chain $X^t$ with stationary distribution being the Gibbs ensemble $P_{n, \beta}$.

\medskip

{\em (i)} Select a vertex $i$ from the underlying graph uniformly, 

\smallskip

{\em (ii)} Update the spin at vertex $i$ according to the distribution $P_{n, \beta}$, conditioned on the event that the spins at all vertices not equal to $i$ remain unchanged.  
\end{defn}

\noi
For more on Glauber dynamics, see \cite{Bre}.

\medskip

An important question of mixing times of dynamics of statistical mechanical models is its relationship with the thermodynamic phase transition structure of the system.  More specifically, as a system undergoes an equilibrium phase transition with respect to some parameter; e.g. temperature, how do the corresponding mixing times behave?  The answer to this question depends on the type of thermodynamic phase transition exhibited by the model.  In the next section we define the two types of thermodynamic phase transition via the large deviation principle of the macroscopic quantity.

\bigskip

\section{Large Deviations and Equilibrium Macrostate Phase Transitions} \label{LDP} 

\medskip

The application of the aggregate path coupling method to prove rapid mixing takes advantage of large deviations estimates that these models satisfy.  In this section, we give a brief summary of large deviations theory used in this paper, written in the context of Gibbs ensembles defined in the previous section.  For a more complete theory of large deviations see for example \cite{DZ} and \cite{Ellis}.

A function $I$ on $\R^q$ is called a {\bf rate function} if $I$ maps $\R^q$ to $[0, \infty]$ and has compact level sets.  

\begin{defn}
Let $I_\beta$ be a rate function on $\R^q$.  The sequence $\{M_n\}$ with respect to the Gibbs ensemble $P_{n, \beta}$ is said to satisfy the {\bf large deviation principle} {\em (LDP)} on $\R^q$ with rate function $I_\beta$ if the following two conditions hold.

For any closed subset $F$,
\be
\label{eqn:upperldp}
\limsup_{n \goto \infty} \frac{1}{n} \log P_{n, \beta} \{ M_n \in F \} \leq - I_\beta (F) 
\ee
and for any open subset $G$, 
\be
\label{lowerldp} 
\liminf_{n \goto \infty} \frac{1}{n} \log P_{n, \beta} \{ M_n \in G \} \geq - I_\beta (G)
\ee
where $I_\beta (A) = \inf_{z \in A} I_\beta (z)$.
\end{defn}

The LDP upper bound in the above definition implies that values $z$
satisfying $I_\beta(z) > 0$ have an exponentially small probability
of being observed as $n \goto \infty$. Hence we define 
the set of {\bf equilibrium macrostates} of the system by
\[
\mathcal{E}_\beta = \{z : I_\beta(z) = 0\}.
\]

For the class of Gibbs ensembles studied in this survey paper, the set of equilibrium macrostates exhibits the following general behavior.  There exists a phase transition critical value of the parameter $\beta_c$ such that 

\smallskip

(a) for $0 < \beta < \beta_c$, the set $\mathcal{E}_\beta$ consists of a single equilibrium macrostate (single phase); i.e. 
\[ \mathcal{E}_\beta = \{\tilde{z}_\beta \} \]

\smallskip

(b) for $\beta_c < \beta$, the set $\mathcal{E}_\beta$ consists of a multiple equilibrium macrostates (multiple phase); i.e.
\[ \mathcal{E}_\beta = \{z_{\beta, 1}, z_{\beta, 2}, \ldots, z_{\beta, q}   \} \]

\smallskip

\noi
The transition from the single phase to the multiple phase follows one of two general types.

\smallskip

Continuous, second-order phase transition: For all $j=1, 2, \ldots, q$, $\lim_{\beta \goto \beta_c^+} z_{\beta, j} = \tilde{z}_\beta$

\smallskip

Discontinuous, first-order phase transition: For all $j=1, 2, \ldots, q$, $\lim_{\beta \goto \beta_c^+} z_{\beta, j} \neq \tilde{z}_\beta$

\smallskip

\noi
As mentioned in the Introduction, understanding the relationship between the mixing times of the Glauber dynamics and the equilibrium phase transition structure of the corresponding Gibbs ensembles is a major motivation for the work discussed in this paper.

\vskip 0.1 in
\noindent

Recent rigorous results for statistical mechanical models that undergo continuous, second-order phase transitions, like the famous Ising model, have been published in \cite{LPW, LLP, DLP1}.  For these models, it has been shown that the mixing times undergo a transition at precisely the thermodynamic phase transition point.  In order to show rapid mixing in the subcritical parameter regime ($\beta < \beta_c$) for these models, the classical path coupling method can be applied directly.

\vskip 0.1 in
\noindent

However, for models that exhibit the other type of phase transition: discontinuous, first-order; e.g. Potts model with $q > 2$ \cite{Wu, CET} and the Blume-Capel model \cite{Blu, BEG, Cap1, Cap2, Cap3, EOT} with weak interaction, the mixing time transition does not coincide with the thermodynamic equilibrium phase transition.

\vskip 0.1 in
\noindent

Discontinuous, first-order phase transitions are more intricate than their counterpart, which makes rigorous analysis of these models traditionally more difficult.  Furthermore, the more complex phase transition structure causes certain parameter regimes of the models to fall outside the scope of standard mixing time techniques including the classical path coupling method discussed in subsection \ref{pathcoupling}.  This was the motivation for the development of the aggregate path coupling method.

\vskip 0.1 in
\noindent

In the following two sections, we define and characterize the aggregate path coupling method to two distinct classes of statistical mechanical spin models.  
We begin with the mean-field Blume-Capel (BC) model, a model ideally suited for the analysis of the relationship between the thermodynamic equilibrium behavior and mixing times due to its intricate phase transition structure.  Specifically, the phase diagram of the BC model includes a curve at which the model undergoes a second-order, continuous phase transition, a curve where the model undergoes a first-order, discontinuous phase transition, and a tricritical point which separates the two curves.  Moreover, the BC model clearly illustrates the strength of the aggregate path coupling method within the simpler setting where the macroscopic quantity for the model is one dimensional.

\vskip 0.1 in
\noindent

In section \ref{genspinmodels}, we generalize the ideas applied to the BC model and define the aggregate path coupling method to a large class of statistical mechanical models with macroscopic quantities in arbitrary dimensions.  We end the survey paper with new mixing time results for Glauber dynamics that converge to the so-called generalized Potts model on the complete graph \cite{JKRW} by applying the general aggregate path coupling method derived in that section .

\bigskip

\section{Mean-field Blume-Capel model} \label{MFBC} 

\medskip

The Hamiltonian function on the configuration space $\Lambda^n = \{ -1, 0, 1\}^n$ for the mean-field Blume-Capel model is defined by 
\[ H_{n, K}(\omega) = \sum_{j=1}^n \omega_j^2 - \frac{K}{n} \left( \sum_{j=1}^n \omega_j \right)^2 \]
for configurations $\omega = (\omega_1, \ldots, \omega_n)$. Here $K$ represents the interaction strength of the model.  Then for inverse temperature $\beta$, the mean-field Blume-Capel model is defined by the sequence of probability measures
\[ P_{n, \beta, K}(\omega) = \frac{1}{Z_n(\beta,K)} \exp \left[-\beta H_{n,K}(\omega) \right] \]
where $Z_n(\beta, K)= \sum_{\omega \in \Lambda^n} \exp [-\beta H_{n, K}(\omega)] $ is the normalizing constant called the {\it partition function}.

\medskip

\subsection{Equilibrium Phase Structure}\label{eps}

In \cite{EOT}, using large deviation theory \cite{EHT}, the authors proved the phase transition structure of the BC model.  The analysis of $P_{n, \beta, K}$ was facilitated by expressing it
 in the form of a Curie-Weiss (mean-field Ising)-type model.  This is done by absorbing 
the noninteracting component of the Hamiltonian into the product measure $P_n$ that assigns the probability $3^{-n}$ to each $\omega \in \Lambda^n$, obtaining
\be
\label{eqn:rewritecanon}
P_{n,\beta,K}(d\omega) = 
\frac{1}{\tilde{Z}_n(\beta,K)} \cdot \exp\!\left[ n \beta K 
\!\left(\frac{S_n(\omega)}{n} \right)^2 \right] P_{n,\beta}(d\omega)
\ee
In this formula $S_n(\omega)$ equals
the total spin $\sum_{j=1}^n \omega_j$,
$P_{n,\beta}$ is the product measure on $\Lambda^n$ with identical one-dimensional marginals
\be
\label{eqn:rhobeta}
\rho_\beta(d \omega_j) = 
\frac{1}{Z(\beta)} \cdot \exp(-\beta \omega_j^2) \, \rho(d \omega_j),
\ee
$Z(\beta)$ is the normalizing constant 
$\int_\Lambda \exp(-\beta \omega_j^2) \rho(d \omega_j) = 1 + 2 e^{-\beta}$,
and $\tilde{Z}_n(\beta,K)$ is the normalizing constant
$[Z(\beta)]^n/Z_n(\beta,K)$. 

\vskip 0.1 in
\noindent
Although $P_{n,\beta,K}$ has the form of a Curie-Weiss (mean-field Ising) model
when rewritten as in (\ref{eqn:rewritecanon}), 
it is much more complicated because of the $\beta$-dependent
product measure $P_{n,\beta}$ and the presence of the parameter $K$. 
These complications introduce new features to the BC model described above that are not present in the Curie-Weiss model \cite{Ellis}.

\vskip 0.1 in
\noindent
The starting point of the analysis of the 
phase-transition structure of the BC model is the large deviation
principle (LDP) satisfied by the spin per site or {\it magnetization} $S_n/n$ with respect to $P_{n,\beta,K}$.  
In order to state the form of the rate function, we introduce the cumulant generating function $c_\beta$ of the measure
$\rho_\beta$ defined in (\ref{eqn:rhobeta}); for $t \in \R$ this function is defined by 
\be
\label{eqn:cbeta}
c_\beta(t) = \log \int_\Lambda \exp (t\omega_1) \, \rho_\beta (d\omega_1) = 
\log \!\left[ \frac{1+e^{-\beta}(e^t+e^{-t})}{1+2e^{-\beta}} \right] \nonumber 
\ee
We also introduce the Legendre-Fenchel transform of $c_\beta$, which is defined for $z \in [-1,1]$ by
\[
J_\beta(z) = \sup_{t \in \R} \{tz - c_\beta(t)\}
\]
and is finite for $z \in [-1,1]$.
$J_\beta$ is the rate function in Cram\'{e}r's theorem, which 
is the LDP for $S_n/n$ with respect to the product measures
$P_{n,\beta}$ \cite[Thm.\ II.4.1]{Ellis} and is one of the components of
the proof of the LDP for $S_n/n$ with respect to the BC model $P_{n,\beta,K}$.
This LDP is stated in the next theorem and is proved in Theorem 3.3 in \cite{EOT}.

\begin{thm}
\label{thm:ldppnbetak}  
For all $\beta > 0$ and $K > 0$, with respect to $P_{n,\beta,K}$,
$S_n/n$ satisfies the large deviation principle on $[-1,1]$ with exponential speed $n$ and rate function
\[
I_{\beta,K}(z) = J_\beta(z) - \beta K z^2 - \inf_{y \in \R}
\{J_\beta(y) - \beta K y^2 \}.
\]
\end{thm}
\noi
The LDP in the above theorem implies that those $z \in [-1,1]$
satisfying $I_{\beta,K}(z) > 0$ have an exponentially small probability
of being observed as $n \goto \infty$. Hence we define 
the set of equilibrium macrostates by
\[
\tilde{\mathcal{E}}_{\beta, K} = \{z \in [-1,1] : I_{\beta,K}(z) = 0\}.
\]

\vskip 0.1 in
\noindent
For $z \in \R$ we define
\be
\label{eqn:gbetak} 
G_{\beta,K}(z) = \beta K z^2 - c_\beta(2\beta K z)
\ee
and as in \cite{EMO1} and \cite{EMO2} refer to it as the {\it free energy functional} of the model.  The calculation of the zeroes of $I_{\beta,K}$ --- equivalently, the global minimum points
of $J_{\beta,K}(z) - \beta K z^2$ --- is greatly facilitated by the following observations
made in Proposition 3.4 in \cite{EOT}: 
\begin{enumerate}
\item 
The global minimum points of 
$J_{\beta,K}(z) - \beta K z^2$ coincide with the global minimum points of $G_{\beta,K}$,
which are much easier to calculate. 

\item 
The minimum values $\min_{z \in \R}\{J_{\beta,K}(z) - \beta K z^2\}$ 
and $\min_{z \in \R}\{G_{\beta,K}(z)\}$ coincide.
\end{enumerate}
Item (1) gives the alternate characterization that 
\be
\label{eqn:ebetak}
\tilde{\mathcal{E}}_{\beta, K} = \{z \in [-1,1] : z \mbox{ minimizes } G_{\beta,K}(z)\}.
\ee

\vskip 0.1 in
\noindent
The free energy functional $G_{\beta, K}$ exhibits two distinct behaviors depending on whether $\beta \leq \beta_c = \log 4$ or $\beta > \beta_c$.  In the first case, the behavior is similar to the Curie-Weiss (mean-field Ising) model.  Specifically, there exists a critical value $K_c^{(2)}(\beta)$ defined in (\ref{eqn:kcbeta}) such that for $K < K_c^{(2)}(\beta)$, $G_{\beta, K}$ has a single minimum point at $z = 0$.  At the critical value $K = K_c^{(2)}(\beta)$, $G_{\beta, K}$ develops symmetric non-zero minimum points and a local maximum point at $z =0$.  This behavior corresponds to a continuous, second-order phase transition and is illustrated in Figure \ref{continuous}.

\begin{figure}[h]
\begin{center}
\includegraphics[height=1.6in]{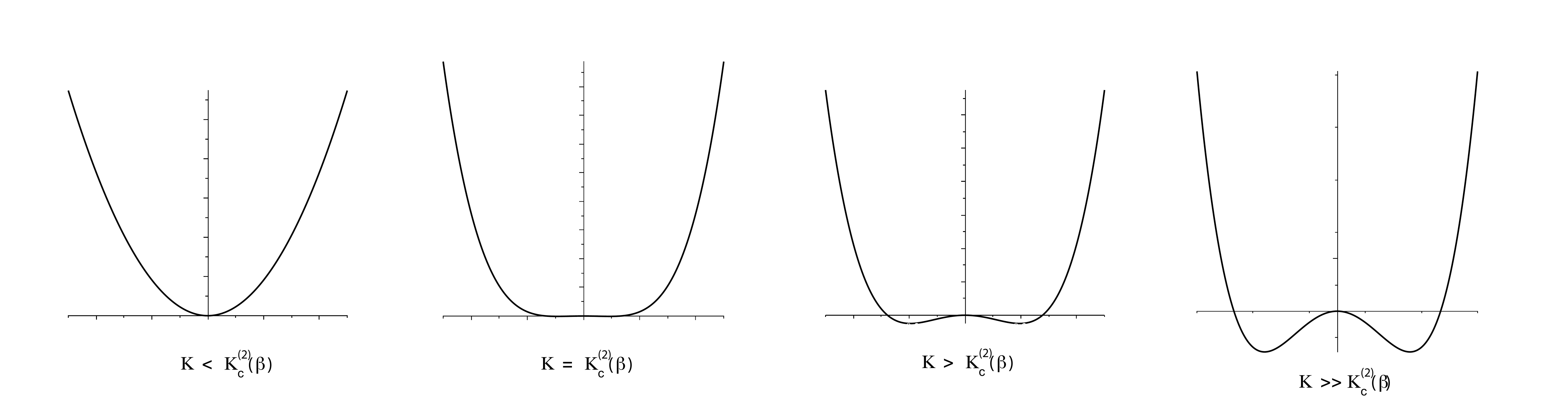}
\caption{\footnotesize The free-energy functional $G_{\beta, K}$ for $\beta \leq \beta_c$} \label{continuous}
\end{center}
\end{figure} 

\vskip 0.1 in
\noindent
On the other hand, for $\beta > \beta_c$, $G_{\beta, K}$ undergoes two transitions at the values denoted by $K_1(\beta)$ and $K_c^{(1)}(\beta)$.  For $K < K_1(\beta)$, $G_{\beta, K}$ again possesses a single minimum point at $z = 0$.  At the first critical value $K_1(\beta)$, $G_{\beta, K}$ develops symmetric non-zero local minimum points in addition to the global minimum point at $z=0$.  These local minimum points are referred to as {\it metastable states} and we refer to $K_1(\beta)$ as the {\it metastable critical value}.  This value is defined implicitly in Lemma 3.9 of \cite{EOT} as the unique value of $K$ for which there exists a unique $z >0$ such that
\[  G_{\beta, K_1(\beta)}'(z) = 0 \hsp \mbox{and} \hsp G_{\beta, K_1(\beta)}''(z) = 0 \] 
As $K$ increases from $K_1(\beta)$ to $K_c^{(1)}(\beta)$, the local minimum points decrease until at $K=K_c^{(1)}(\beta)$, the local minimum points reach zero and $G_{\beta, K}$ possesses three global minimum points.  Therefore, for $\beta > \beta_c$, the BC model undergoes a phase transition at $K=K_c^{(1)}(\beta)$, which is defined implicitly in \cite{EOT}.  Lastly, for $K > K_c^{(1)}(\beta)$, the symmetric non-zero minimum points drop below zero and thus $G_{\beta, K}$ has two symmetric non-zero global minimum points.  This behavior corresponds to a discontinuous, first-order phase transition 
 and is illustrated in Figure \ref{discont}.

\begin{figure}[h]
\begin{center}
\includegraphics[height=2.75in]{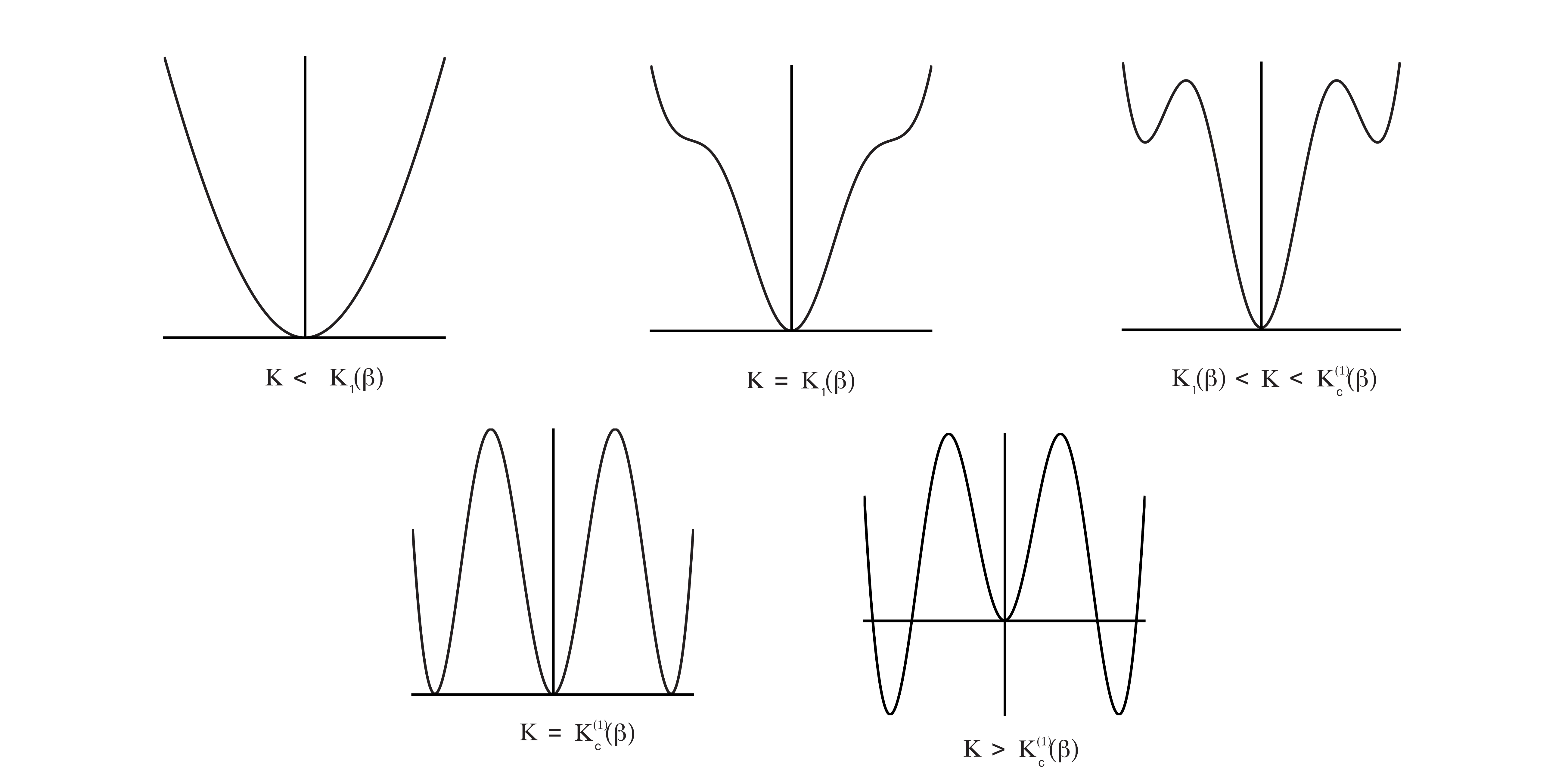}
\caption{\footnotesize The free-energy functional $G_{\beta, K}$ for $\beta > \beta_c$} \label{discont}
\end{center}
\end{figure} 

\vskip 0.1 in
\noindent
In the next two theorems, the 
structure of $\tilde{\mathcal{E}}_{\beta, K}$ corresponding to the behavior of $G_{\beta, K}$ just described is stated which depends on the relationship between $\beta$ and the
critical value $\beta_c = \log 4$.
We first describe $\tilde{\mathcal{E}}_{\beta, K}$ for 
$0 < \beta \leq \beta_c$ and then for $\beta > \beta_c$.  
In the first case $\tilde{\mathcal{E}}_{\beta, K}$ undergoes a continuous
bifurcation as $K$ increases through the critical value $K_c^{(2)}(\beta)$
defined in (\ref{eqn:kcbeta}); physically, this bifurcation
corresponds to a second-order phase transition.  The following theorem 
is proved in Theorem 3.6 in \cite{EOT}.

\begin{thm}
\label{thm:secondorder} 
For $0 < \beta \leq \beta_c$, we define
\be
\label{eqn:kcbeta}
K_c^{(2)}(\beta) = \frac{1}{2\beta c''_\beta(0)} =
\frac{e^\beta + 2}{4\beta}.
\ee
For these values of $\beta$, $\tilde{\mathcal{E}}_{\beta, K}$ has the following structure.

{\em(a)} For $0 < K \leq K_c^{(2)}(\beta)$,
$\tilde{\mathcal{E}}_{\beta,K} = \{0\}$.

{\em(b)} For $K > K_c^{(2)}(\beta)$, there exists 
${z}(\beta,K) > 0$ such that
$\tilde{\mathcal{E}}_{\beta,K} = \{\pm z(\beta,K) \}$.

{\em(c)} ${z}(\beta,K)$ is
a positive, increasing, continuous function for $K > K_c^{(2)}(\beta)$, and
as $K \goto (K_c^{(2)}(\beta))^+$, $z(\beta,K) \goto 0$. 
Therefore, $\tilde{\mathcal{E}}_{\beta,K}$ exhibits a continuous bifurcation
at $K_c^{(2)}(\beta)$.
\end{thm}
\noi
For $\beta \in (0,\beta_c)$, the curve $(\beta,K_c^{(2)}(\beta))$ is the curve of second-order 
critical points.   As we will see in a moment, for $\beta \in (\beta_c,\infty)$
the BC model also has a curve
of first-order critical points, which we denote by $(\beta,K_c^{(1)}(\beta))$.

\vskip 0.1 in
\noindent
We now describe $\tilde{\mathcal{E}}_{\beta, K}$ for 
$\beta > \beta_c$.  In this case $\tilde{\mathcal{E}}_{\beta, K}$ undergoes a discontinuous
bifurcation as $K$ increases through an implicitly defined critical value.
Physically, this bifurcation
corresponds to a first-order phase transition.  The following theorem 
is proved in Theorem 3.8 in \cite{EOT}.  

\begin{thm}
\label{thm:firstorder} 
For all $\beta > \beta_c $, $\tilde{\mathcal{E}}_{\beta, K}$ has
the following structure in terms of the quantity $K_c^{(1)}(\beta)$ defined implicitly for $\beta > \beta_c$ on page {\em 2231} of
{\em \cite{EOT}}.

{\em(a)} For $0 < K < K_c^{(1)}(\beta)$,
$\tilde{\mathcal{E}}_{\beta,K} = \{0\}$.

{\em(b)} There exists $z(\beta,K_c^{(1)}(\beta)) > 0$
such that $\tilde{\mathcal{E}}_{\beta,K_c^{(1)}(\beta)} =
\{0,\pm z(\beta,K_c^{(1)}(\beta))\}$.

{\em(c)} For $K > K_c^{(1)}(\beta)$ 
there exists $z(\beta,K) > 0$
such that $\tilde{\mathcal{E}}_{\beta,K} =
\{\pm z(\beta,K)\}$.

{\em(d)} $z(\beta,K)$
is a positive, increasing, continuous function for $K \geq K_c^{(1)}(\beta)$, and 
as $K \goto K_c^{(1)}(\beta)^+$, $z(\beta,K) \goto 
z(\beta,K_c^{(1)}(\beta)) > 0$.  Therefore,
$\tilde{\mathcal{E}}_{\beta,K}$ exhibits a discontinuous bifurcation
at $K_c^{(1)}(\beta)$.
\end{thm}
\noi
The phase diagram of the BC model is depicted in Figure \ref{phase}.  The LDP stated in Theorem \ref{thm:ldppnbetak} implies the following weak convergence result used in the proof of rapid mixing in the first-order, discontinuous phase transition region.  It is part (a) of Theorem 6.5 in \cite{EOT}.

\begin{thm}
\label{thm:weakconv}
For $\beta$ and $K$ for which $\tilde{\mathcal{E}}_{\beta, K} = \{ 0 \}$, 
\[ P_{n, \beta, K} \{ S_n/n \in dx \} \Longrightarrow \delta_0 \hsp \mbox{as} \ \ n \goto \infty. \]
\end{thm}

\begin{figure}[t]
\begin{center}
\includegraphics[height=2.7in]{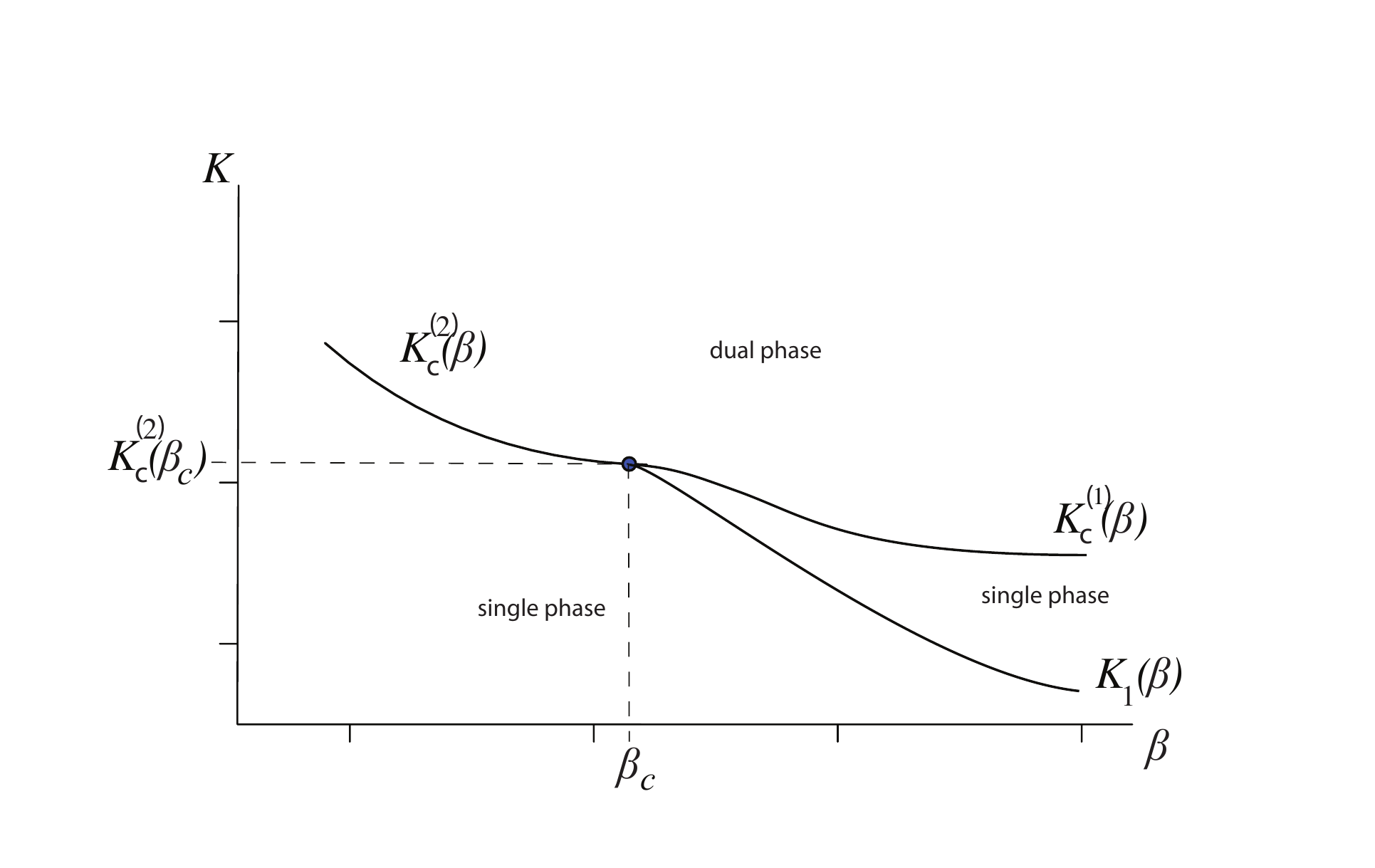}
\caption{\footnotesize Equilibrium phase transition structure of the mean-field Blume-Capel model} \label{phase}
\end{center}
\end{figure} 

\vskip 0.1 in
\noindent
We end this section with a final result that was not included in the original paper \cite{EOT} but will be used in the proof of the slow mixing result for the BC model.  The result states that not only do the global minimum point of $G_{\beta, K}$ and $I_{\beta, K}$ coincide, but so do the local minimum points.

\begin{lemma}
\label{lemma:equalmin}
In the case where $G_{\beta, K}$
and $I_{\beta, K}$ are strictly convex at their minimum points, a point $\tilde{z}$ is a local minimum point of $G_{\beta, K}$ if and only if it is a local minimum point of $I_{\beta, K}$.
\end{lemma}
\bp
Assume that $\tilde{z}$ is a local minimum point of $G_{\beta, K}$.  Then $\tilde{z}$ is a critical point of $G_{\beta, K}$ which implies that $\tilde{z} = c_\beta '(2 \beta K \tilde{z})$.  By the theory of Legendre-Fenchel transforms, $J_\beta '(z) = (c_\beta ')^{-1}(z)$ and thus 
\[ I_{\beta, K}'(\tilde{z}) = J_{\beta} '(\tilde{z}) - 2 \beta K \tilde{z} = (c_\beta ')^{-1}(\tilde{z}) - 2 \beta K \tilde{z} = 0. \]

\vskip 0.1 in
\noindent
Next, since $\tilde{z}$ is a local minimum point of $G_{\beta, K}$, 
\[ G_{\beta, K} ''(\tilde{z}) > 0 \hsp \Longleftrightarrow \hsp c_{\beta}''(2 \beta K \tilde{z}) < \frac{1}{2 \beta K} \]
Therefore,
\[ I_{\beta, K}''(\tilde{z}) = J_{\beta}''(\tilde{z}) - 2 \beta K = \frac{1}{c_{\beta}''(2 \beta K \tilde{z})} - 2 \beta K > 0 \]
and we conclude that $\tilde{z}$ is a local minimum point of $I_{\beta, K}$.  The other direction is obtained by reversing the argument.
\ep

\medskip

\subsection{Glauber Dynamics}  \label{mt}

The Glauber dynamics, defined in general in section \ref{gibbs}, for the mean-field Blume-Capel model evolve by selecting a vertex $i$ at random and updating the spin at $i$ according to the distribution $P_{n, \beta, K}$, conditioned to agree with the spins at all vertices not equal to $i$.  If the current configuration is $\omega$ and vertex $i$ is selected, then the chance of the spin at $i$ is updated to $+1$ is equal to 
\be
\label{eqn:plustran}
p_{+1} (\omega, i) = \frac{e^{2\beta K \tilde{S}(\omega, i)/n }}{e^{2\beta K \tilde{S}(\omega, i)/n} + e^{\beta-(\beta K)/n} + e^{-2\beta K \tilde{S}(\omega, i)/n} }  \ee
where $\tilde{S}(\omega, i) = \sum_{j\neq i} \omega_j$ is the total spin of the neighboring vertices of $i$.
Similarly, the probabilities of $i$ updating to $0$ and $-1$ are
\be 
\label{eqn:zerotran}
p_{0}(\omega, i) = \frac{e^{\beta-(\beta K)/n}}{e^{2\beta K \tilde{S}(\omega, i)/n} + e^{\beta-(\beta K)/n} + e^{-2\beta K \tilde{S}(\omega, i)/n} } \ee
and
\be 
\label{eqn:minustran}
p_{-1}(\omega, i) = \frac{e^{-2\beta K \tilde{S}(\omega, i)/n} }{e^{2\beta K \tilde{S}(\omega, i)/n} + e^{\beta-(\beta K)/n} + e^{-2\beta K \tilde{S}(\omega, i)/n} }
\ee
$p_{+1} (\omega, i)$ is increasing with respect to $\tilde{S}(\omega, i)$, $p_{-1} (\omega, i)$ is decreasing with respect to $\tilde{S}(\omega, i)$, and $p_{0} (\omega, i)$ is decreasing for $\tilde{S}(\omega, i) > 0$ and increasing for $\tilde{S}(\omega, i) < 0$.

\medskip

A classical tool in proving rapid mixing for Markov chains defined on graphs, including the Glauber dynamics of statistical mechanical models, is the path coupling technique discussed in subsection \ref{pathcoupling}.  It will be shown that this technique can be directly applied to the BC model in the second-order, continuous phase transition region but fails in a subset of the first-order, discontinuous phase transition region.  It is for the latter region that we developed the {\it aggregate path coupling}  method to prove rapid mixing.  First, the standard path coupling method for the BC model is introduced in the next section.

\medskip

\subsection{Path Coupling}  \label{pc}

We begin by setting up the coupling rules for the Glauber dynamics of the mean-field Blume-Capel model.  Define the path metric $\rho$ on $\Omega^n = \{-1, 0, 1\}^n$ by 
\be
\label{eqn:pathmetricBC} 
\rho(\sigma, \tau) = \sum_{j=1}^n \big|\sigma_j - \tau_j \big|.
\ee

\begin{remark}
In the original paper \cite{KOT} on the mixing times of the mean-field Blume-Capel model, the incorrect path metric was used.  In that paper, the path metric was defined by 
\[ \tilde{\rho}(\sigma, \tau) = \sum_{j=1}^n \bf{1} \{\sigma_j \neq \tau_j \} \]
With the correct metric defined in (\ref{eqn:pathmetricBC}), the proofs in \cite{KOT} remains the same.
\end{remark}

\medskip

Let $\sigma$ and $\tau$ be two configurations with $\rho(\sigma, \tau) = 1$; i.e. $\sigma$ and $\tau$ are neighboring configurations.  The spins of $\sigma$ and $\tau$ agree everywhere except at a single vertex $i$, where either $\sigma_i=0$ and $\tau_i \not=0$, or $\sigma_i \not=0$ and $\tau_i =0$. Assume that $\sigma_i=0$ and  $\tau_i=1$.  We next describe the path coupling $(X, Y)$ of one step of the Glauber dynamics starting in configuration $\sigma$ with one starting in configuration $\tau$.
Pick a vertex $k$ uniformly at random.  We use a single random variable as the common source of noise to update both chains, so the two chains agree as often as possible.  In particular, let $U$ be a uniform random variable on $[0,1]$ and set
\[ X(k) = \left\{ \begin{array}{rl} -1 & \mbox{if \ $0 \leq U \leq p_{-1}(\sigma, k)$} \\ 0 & \mbox{if \ $p_{-1}(\sigma, k) < U \leq p_{-1}(\sigma, k) + p_{0}(\sigma, k)$} \\ +1 & \mbox{if \ $p_{-1}(\sigma, k) + p_{0}(\sigma, k) < U \leq 1$} \end{array} \right. \]
and
\[ Y(k) = \left\{ \begin{array}{rl} -1 & \mbox{if \ $0 \leq U \leq p_{-1}(\tau, k)$} \\ 0 & \mbox{if \ $p_{-1}(\tau, k) < U \leq p_{-1}(\tau, k) + p_{0}(\tau, k)$} \\ +1 & \mbox{if \ $p_{-1}(\tau, k) + p_{0}(\tau, k) < U \leq 1$} \end{array} \right. \]
Set $X(j) = \sigma_j$ and $Y(j) = \tau_j$ for $j \neq k$.

\vskip 0.2 in
\noindent
Since $\sigma_i < \tau_i$, for all $j \neq i$, $\tilde{S}(\sigma, j) < \tilde{S}(\tau, j)$ and thus
\[ p_{+1}(\tau, k) > p_{+1}(\sigma, k) \ \ \ \mbox{and} \ \ \ p_{-1}(\tau, k) < p_{-1}(\sigma, k) \]
The path metric $\rho$ on the coupling above takes on the following possible values.
\[ \rho (X,Y) = \left\{ \begin{array}{lll} 0 & \mbox{if} & k = i \\ 1 & \mbox{if} & \mbox{$k \neq i$ and both chains updates the same}
\\ 2 & \mbox{if} & \mbox{$k \neq i$ and the chains update differently} \end{array} \right.  \]
Note that since $\sigma$ and $\tau$ are neighbor configurations, $\rho(X,Y) \neq 3$ because the update probabilities of $X$ and $Y$ are sufficiently close.

\medskip

The application of the path coupling technique to prove rapid mixing is dependent on whether the mean coupling distance with respect to the path metric $\rho$, denoted by $\mathbb{E}_{\sigma, \tau} [ \rho(X,Y)] $, contracts over {\it all} pairs of neighboring configurations.

\vskip 0.2 in
\noindent
In the lemma below and following corollary, we derive a working form for the mean coupling distance.

\begin{lemma}
\label{lemma:meanpath}
Let $\rho$ be the path metric defined in {\em (\ref{eqn:pathmetric})} and $(X,Y)$ be the path coupling of one step of the Glauber dynamics of the mean-field Blume-Capel model where $X$ and $Y$ start in neighboring configurations $\sigma$ and $\tau$.  Define 
\be
\label{eqn:varphi} 
\varphi_{\beta, K}(x) = \frac{2\sinh (\frac{2\beta K}{n}x)}{ 2\cosh (\frac{2\beta K}{n} x) + e^{\beta - \frac{\beta K}{n}}} 
\ee
Then 
\[ \mathbb{E}_{\sigma, \tau} [ \rho(X,Y)]  = \frac{n-1}{n} + \frac{(n-1)}{n} [\varphi_{\beta, K}(S_n(\tau)) - \varphi_{\beta, K}(S_n(\sigma))] + O\left(\frac{1}{n^2} \right) \]
\end{lemma}
\bp
Let $n_{-1}, n_0$ and $n_{+1}$ denote the number of $-1, 0$ and $+1$ spins, respectively, in configuration $\sigma$, not including the spin at vertex $i$, where the configurations differ.  Note that $n_{-1} + n_0 + n_{+1} = n-1$.  

\vskip 0.2 in
\noindent
Define $\varepsilon(-1)$ to be the probability that $X$ and $Y$ update differently when the chosen vertex $k \neq i$ is a $-1$ spin.  Similarly, define $\varepsilon(0)$ and $\varepsilon(+1)$.  Then the mean coupling distance can be expressed as
\beas
\mathbb{E}_{\sigma, \tau} [ \rho(X,Y)] & = & \frac{n_{-1}}{n} (1 - \varepsilon(-1)) +  \frac{n_{0}}{n} (1 - \varepsilon(0)) +  \frac{n_{+1}}{n} (1 - \varepsilon(+1)) \\
 & & + 2 \left[ \frac{n_{-1}}{n} \varepsilon(-1) +  \frac{n_{0}}{n} \varepsilon(0) +  \frac{n_{+1}}{n} \varepsilon(+1) \right] \\
 & = & \frac{n-1}{n} + \frac{n_{-1}}{n} \varepsilon(-1) +  \frac{n_{0}}{n} \varepsilon(0) +  \frac{n_{+1}}{n} \varepsilon(+1)
\eeas
The probability that $X$ and $Y$ update differently when the chosen vertex $k \neq i$ is a $-1$ spin is given by
\beas 
\ve(-1) & = & \left[ p_{-1}(\sigma, k) - p_{-1}(\tau, k) \right] + \left[(p_{-1}(\sigma, k) + p_{0}(\sigma, k)) - (p_{-1}(\tau, k) + p_{0}(\tau, k)) \right] \\
& = & [ p_{+1}(\tau, k) - p_{+1}(\sigma, k) ] + \left[ p_{-1}(\sigma, k) - p_{-1}(\tau, k) \right] \\
& = & [ p_{+1}(\tau, k) - p_{-1}(\tau, k) ] + \left[ p_{-1}(\sigma, k) -  p_{+1}(\sigma, k) \right] \\
& = & \frac{2\sinh (\frac{2\beta K}{n}(S_n(\tau)+1))}{ 2\cosh (\frac{2\beta K}{n} (S_n(\tau)+1)) + e^{\beta - \frac{\beta K}{n}}} - \frac{2\sinh (\frac{2\beta K}{n}(S_n(\sigma)+1))}{ 2\cosh (\frac{2\beta K}{n} (S_n(\sigma)+1)) + e^{\beta - \frac{\beta K}{n}}} \\
& = & \varphi_{\beta, K} ((S_n(\tau) +1)) - \varphi_{\beta, K} ((S_n(\sigma) +1))\\
& = & \varphi_{\beta, K} (S_n(\tau) ) - \varphi_{\beta, K} (S_n(\sigma))+ O\left(\frac{1}{n^2} \right)
\eeas
Similarly, we have 
$$\ve(0) =\varphi_{\beta, K} (S_n(\tau) ) - \varphi_{\beta, K} (S_n(\sigma))$$
 and 
 $$\ve(+1) = \varphi_{\beta, K} ((S_n(\tau) -1)) - \varphi_{\beta, K} ((S_n(\sigma) -1)) =  \varphi_{\beta, K} (S_n(\tau) ) - \varphi_{\beta, K} (S_n(\sigma))+ O\left(\frac{1}{n^2} \right)$$  
and the proof is complete.
\ep

\noindent
For $c_\beta$ defined in (\ref{eqn:cbeta}), we have
\[ \varphi_{\beta,K}(x)= c_\beta ' \left( \frac{2 \beta K}{n} x \right) (1+O(1/n)) \]
which yields the following corollary.

\medskip

\begin{cor} \label{cor:meanpath}
Let $\rho$ be the path metric defined in {\em (\ref{eqn:pathmetric})} and $(X,Y)$ be the path coupling where $X$ and $Y$ start in neighboring configurations $\sigma$ and $\tau$.  Then 
\[ \mathbb{E}_{\sigma, \tau} [ \rho(X,Y)]  = \frac{n-1}{n} + \frac{(n-1)}{n} \left[c_\beta ' \left( 2\beta K \frac{S_n(\tau)}{n} \right) - c_\beta ' \left( 2 \beta K \frac{S_n(\sigma)}{n} \right) \right] + O\left(\frac{1}{n^2} \right). \]
\end{cor}

\noi
By the above corollary, we conclude that the mean coupling distance of a coupling starting in neighboring configurations contracts; i.e. $\mathbb{E}_{\sigma, \tau}[\rho(X,Y)] < \rho(\sigma, \tau) = 1$, if 

\[ \left[c_\beta ' \left( 2\beta K \frac{S_n(\tau)}{n} \right) - c_\beta ' \left( 2 \beta K \frac{S_n(\sigma)}{n} \right) \right] \approx 2 \beta K \left[ \frac{S_n(\tau)}{n} - \frac{S_n(\sigma)}{n} \right] c_\beta ''\left(2 \beta K \frac{S_n(\sigma)}{n} \right) < \frac{1}{n-1} \]

\noi
Since $\sigma$ and $\tau$ are neighboring configurations and $S_n (\tau) > S_n(\sigma)$, this is equivalent to 
\be
\label{eqn:cpp} 
c_\beta ''\left(2 \beta K \frac{S_n(\sigma)}{n} \right) < \frac{1}{2 \beta K} 
\ee
Therefore, contraction of the mean coupling distance, and thus rapid mixing, depends on the concavity behavior of the function $c_\beta'$.  This is also precisely what determines the type of thermodynamic equilibrium phase transition (continuous, second-order versus discontinuous, first-order) that is exhibited by the mean-field Blume-Capel model.  We state the concavity behavior of $c_\beta'$ in the next theorem which is proved in Theorem 3.5 in \cite{EOT}.  The results of the theorem are depicted in Figure \ref{cprime}

\begin{thm} 
\label{thm:concavity}
For $\beta > \beta_c = \log 4$ define 
\be
\label{eqn:wcrit}
w_c(\beta) = \cosh^{-1} \! \left( \frac{1}{2} e^\beta -
4e^{-\beta}  \right) \geq 0.  
\ee 
The following conclusions hold.

{\em (a)} For $0 < \beta \leq \beta_c$, $c_\beta'(w)$ is strictly concave  for
$w > 0$.

{\em (b)} For $\beta > \beta_c$,
$c_\beta'(w)$ is strictly convex for $0 < w < w_c(\beta)$ \ 
and $c_\beta'(w)$ is strictly concave for $w > w_c(\beta)$.
\end{thm} 

\noi
By part (a) of the above theorem, for $\beta \leq \beta_c$, $~~c_\beta''(x)\leq c_\beta''(0) = 1/(2 \beta K_c^{(2)}(\beta))$.  Therefore, by (\ref{eqn:cpp}), the mean coupling distance contracts between {\it all} pairs of neighboring states whenever $K<K_c^{(2)}(\beta)$. 

\vskip 0.2 in
\noi
By contrast, for $\beta>\beta_c$, we will show that rapid mixing occurs whenever $K < K_1(\beta)$ where $K_1(\beta)$ is the metastable critical value introduced in Subsection \ref{eps} and depicted in Figure \ref{discont}.   However, since the supremum  $~~\sup_{[-1,1]} c_\beta''(x)>{1 \over 2\beta K_1(\beta)}$, the condition $K<K_1(\beta)$ is not sufficient for (\ref{eqn:cpp}) to hold. That is, $K<K_1(\beta)$ does not imply the contraction of the mean coupling distance between {\it all} pairs of neighboring states.
However, we prove rapid mixing for all $K<K_1(\beta)$ in Subsection \ref{rapid2} by using an extension to the path coupling method that we refer to as {\it aggregate path coupling}.

\begin{figure}[t]
\begin{center}
\includegraphics[height=2.5in]{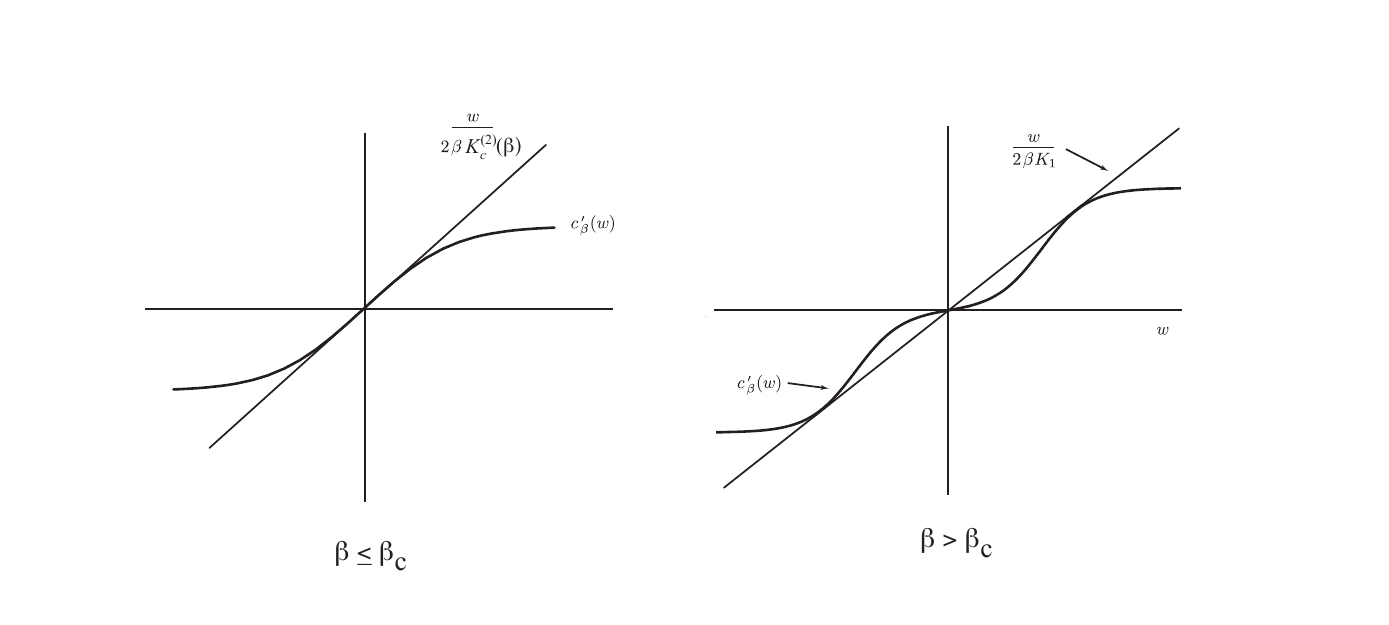}
\caption{\footnotesize Behavior of $c_\beta'(w)$ for large and small $\beta$.} \label{cprime}
\end{center}
\end{figure} 

\vskip 0.2 in
\noi
We now prove the mixing times for the mean-field Blume-Capel model, which varies depending on the parameter values $(\beta, K)$ and their position with respect to the thermodynamic phase transition curves.  We begin with the case $\beta \leq \beta_c$ where the model undergoes a continuous, second-order phase transition and $K \leq K_c^{(2)}(\beta)$ which corresponds to the single phase region.

\medskip

\subsection{Standard Path Coupling in the Continuous Phase Transition Region} \label{rapid1}  

\medskip

We begin by stating the standard path coupling argument used to prove rapid mixing for the mean-field Blume-Capel model in the continuous, second-order phase transition region.  The result is proved in Proposition \ref{classicalPC}.

\begin{thm}
\label{thm:path}
Suppose the state space $\Omega$ of a Markov chain is the vertex set of a graph with path metric $\rho$.  Suppose that for each edge $\{\sigma,\tau\}$ there exists a coupling $(X,Y)$ of the distributions $P(\sigma, \cdot)$ and $P(\tau, \cdot)$ such that 
\[ \mathbb{E}_{\sigma, \tau}[\rho (X,Y)] \leq \rho(\sigma, \tau) e^{-\alpha} \quad \text{ for some }\alpha>0 \]
Then
\[ t_{\text{{\em \mix}}}(\ve) \leq \left\lceil \frac{-\log (\ve) + \log (\mbox{{\em diam}}(\Omega))}{\alpha} \right\rceil \]
\end{thm}

\vskip 0.2 in
\noi
In this section, we assume $\beta \leq \beta_c$ which implies that the BC model undergoes a continuous, second-order phase transition at $K = K_c^{(2)}(\beta)$ defined in (\ref{eqn:kcbeta}).  By Theorem \ref{thm:concavity}, for $\beta \leq \beta_c$, $c_\beta'(x)$ is concave for $x > 0$.  See the first graph of Figure \ref{cprime} as reference. 
We next state and prove the rapid mixing result for the mean-field Blume-Capel model in the second-order, continuous phase transition regime.

\begin{thm}
\label{thm:mfBC}
Let $t_{mix}(\ve)$ be the mixing time for the Glauber dynamics of the mean-field Blume-Capel model on $n$ vertices and $K_c^{(2)}(\beta)$ the continuous phase transition curve defined in {\em (\ref{eqn:kcbeta})}.  Then for $\beta \leq \beta_c = \log 4$ and $K < K_c^{(2)}(\beta)$,
\[ t_{mix}(\ve) \leq {n \over \alpha} (\log n + \log (1/\ve)) \]
for any $\alpha \in \left(0, {K_c^{(2)}(\beta) -K  \over K_c^{(2)}(\beta)}\right)$ and $n$ sufficiently large.
\end{thm}

\begin{proof}
Let $(X,Y)$ be a coupling of the Glauber dynamics of the BC model that begin in neighboring configurations $\sigma$ and $\tau$ with respect to the path metric $\rho$ defined in (\ref{eqn:pathmetric}).  By Corollary \ref{cor:meanpath} of Lemma \ref{lemma:meanpath}, 
\[
\mathbb{E}_{\sigma, \tau} [ \rho(X,Y)]  = 1 - \left( \frac{1}{n} - \frac{(n-1)}{n} \left[c_\beta ' \left( 2\beta K \frac{S_n(\tau)}{n} \right) - c_\beta ' \left( 2 \beta K \frac{S_n(\sigma)}{n} \right) \right] \right) + O\left(\frac{1}{n^2} \right) \]
Observe that $c_\beta ''$ is an even function and that for $\beta \leq \beta_c$, ${\dstyle \sup_x c_\beta ''(x) = c_\beta ''(0)}$.  Therefore, by the mean value theorem and Theorem \ref{thm:secondorder},
\beas 
\mathbb{E}_{\sigma, \tau} [ \rho(X,Y)] & \leq & 1 - \frac{[1 - (n-1)(2\beta K /n) c_\beta ''(0)]}{n} + O\left(\frac{1}{n^2} \right) \\
& \leq & \exp \left\{- \frac{1-2\beta K c_\beta ''(0)}{n} + O\left( \frac{1}{n^2} \right) \right\} \\
& = &   \exp \left\{\frac{1}{n} \left({K_c^{(2)}(\beta) -K  \over K_c^{(2)}(\beta)} \right) + O \left(\frac{1}{n^2} \right) \right\} \\
& < & e^{-\alpha/n}
\eeas
for any $\alpha \in \left(0, {K_c^{(2)} (\beta) -K  \over K_c^{(2)}(\beta)}\right)$ and $n$ sufficiently large.
Thus, for $K < K_c^{(2)}(\beta)$, we can apply Theorem \ref{thm:path}, where the diameter of the configuration space of the BC model $\Omega^n$ is $n$, to complete the proof.
\ep

\medskip

\subsection{Aggregate Path Coupling in the Discontinuous Phase Transition Region} \label{rapid2}  

\medskip

Here we consider the region $\beta>\beta_c$, where the mean-field Blume-Capel model undergoes a first-order discontinuous phase transition.  In this region, the function $c_\beta'(x)$ which determines whether the mean coupling distance contracts (Corollary \ref{cor:meanpath}) is no longer strictly concave for $x > 0$ (Theorem \ref{thm:concavity}).  See the second graph in Figure \ref{cprime} for reference.  We will show that rapid mixing occurs whenever $K < K_1(\beta)$ where $K_1(\beta)$ is the metastable critical value defined in subsection \ref{eps} and depicted in Figure \ref{discont}.

\vskip 0.2 in
\noi
As shown in Section \ref{pc}, in order to apply the standard path coupling technique of Theorem \ref{thm:path}, we need the inequality (\ref{eqn:cpp}) to hold for all values of $S_n(\sigma)$ and thus $~~\sup_{[-1,1]} c_\beta''(x)<{1 \over 2\beta K}$.  However since $~~\sup_{[-1,1]} c_\beta''(x)>{1 \over 2\beta K_1(\beta)}$, the condition $K<K_1(\beta)$ is not sufficient for the contraction of the mean coupling distance between {\it all pairs} of neighboring states which is required to prove rapid mixing using the standard path coupling technique stated in Theorem \ref{thm:path}. 

\vskip 0.2 in
\noi 
In order to prove rapid mixing in the region where $\beta>\beta_c$ and $K<K_1(\beta)$, we take advantage of the result in Theorem \ref{thm:weakconv} which states the weak convergence of the magnetization $S_n/n$ to a point-mass at the origin.   Thus, in the coupling of the dynamics, the magnetization of the process that starts at equilibrium will stay mainly near the origin.  As a result, for two starting configurations $\sigma$ and $\tau$, one of which has near-zero magnetization ($S_n(\sigma)/n \approx 0$), the mean coupling distance of a coupling starting in these configurations will be the aggregate of the mean coupling distances between neighboring states along a minimal path connecting the two configurations. Although not all pairs of neighbors in the path will contract, we show that in the {\it aggregate}, contraction between the two configurations still holds.

\vskip 0.2 in
\noi 
In the next lemma we prove contraction of the mean coupling distance in the aggregate and then the rapid mixing result for the mean-field Blume-Capel model is proved in the theorem following the lemma by applying the new aggregate path coupling method.

\begin{lemma}
\label{lemma:aggregate}
Let $(X,Y)$ be a coupling of one step of the Glauber dynamics of the BC model that begin in configurations $\sigma$ and $\tau$, not necessarily neighbors with respect to the path metric $\rho$ defined in {\em (\ref{eqn:pathmetric})}.  Suppose $\beta > \beta_c$ and $K < K_1(\beta)$. Then for any $\alpha \in \left(0,{K_1(\beta)-K \over K_1(\beta)}\right)$ there exists an $\varepsilon>0$ such that, asymptotically as $n \goto \infty$,
\be
\label{eqn:contract} 
\mathbb{E}_{\sigma, \tau} [ \rho(X,Y)] \leq e^{-\alpha/n} \rho(\sigma, \tau) 
\ee
whenever $|S_n(\sigma) |<\varepsilon n$.
\end{lemma}
\bp
Observe that for $\beta > \beta_c$ and $K < K_1(\beta)$, $$|c_\beta '(x)| \leq {|x| \over  2\beta K_1(\beta)} \hsp \mbox{for all} \ x $$ 
We will show that for a given $\alpha' \in \left({1 \over 2\beta K_1(\beta)},{1-\alpha \over 2\beta K} \right)$, there exists $\varepsilon>0$ such that
\begin{equation}\label{cb}
c_\beta '(x)-c_\beta '(x_0) \leq \alpha' (x-x_0)\quad \text{ whenever}~ |x_0|<\varepsilon
\end{equation}
as $c_\beta '(x)$ is a continuously differentiable increasing odd function and $c_\beta '(0)=0$.

\vskip 0.2 in
\noi 
In order to show (\ref{cb}), observe that $c''_{\beta}(0)={1 \over 2\beta K_c^{(2)}(\beta)} < {1 \over 2 \beta K_1(\beta)}$, and since $c''_{\beta}$ is continuous, there exists a $\delta>0$ such that
$$c_\beta ''(x)<\alpha'  \quad \text{ whenever}~ |x|<\delta$$
The mean value theorem implies that
$$c_\beta '(x)-c_\beta '(x_0) <\alpha' (x-x_0)\qquad \mbox{for all} \  x_0,x \in (-\delta,\delta)$$
Now, let $\varepsilon={\alpha'-1/(2\beta K_1(\beta)) \over \alpha'+1/(2\beta K_1(\beta))} \delta < \delta$.  Then for any $|x_0|<\varepsilon$ and $|x|\geq \delta$,
$$|c_\beta '(x)-c_\beta '(x_0)| \leq {|x|+|x_0| \over 2\beta K_1(\beta)}
\leq {(1+{\varepsilon/\delta})|x| \over 2\beta K_1(\beta)} 
= {|x-x_0| \over 2\beta K_1(\beta)}\cdot {1+{\varepsilon/\delta} \over |1-x_0/x|} 
\leq {|x-x_0| \over 2\beta K_1(\beta)}\cdot {1+{\varepsilon/\delta} \over 1-\varepsilon/\delta}
=\alpha' |x-x_0|$$

\noindent
Without loss of generality suppose that $S_n(\sigma)< S_n(\tau)$.
Let $(\sigma = x_0, x_1, \ldots, x_r = \tau)$ be a path connecting $\sigma$ to $\tau$ and monotone increasing  in $\rho$ such that $(x_{i-1}, x_i)$ are neighboring configurations.  Here $r=\rho(\sigma,\tau)$. Then by Corollary \ref{cor:meanpath} of Lemma \ref{lemma:meanpath} and (\ref{cb}), we have for $|S_n(\sigma) |<\varepsilon n$ and asymptotically as $n \rightarrow \infty$, 
\beas
 \mathbb{E}_{\sigma, \tau} [ \rho(X,Y)]  & \leq & \sum_{i=1}^r \mathbb{E}_{x_{i-1}, x_i} [ \rho(X_{i-1},X_i)] \\
& = &  \frac{(n-1)}{n} \rho(\sigma, \tau) +  \frac{(n-1)}{n} \left[c_\beta'\left( \frac{2\beta K}{n} S_n(\tau) \right) - c_\beta' \left(\frac{2\beta K}{n} S_n(\sigma) \right) \right]  +\rho(\sigma, \tau) \cdot O\left({1 \over n^2}\right)\\
& \leq & \frac{(n-1)}{n} \rho(\sigma, \tau) +  \frac{(n-1)}{n} (S_n(\tau)-S_n(\sigma)) \frac{2\beta K \alpha'}{n}  +\rho(\sigma, \tau) \cdot O\left({1 \over n^2}\right)\\
& \leq & \rho (\sigma, \tau) \left[ 1 - \left(\frac{1-2\beta K \alpha'}{n} \right)  + O\left({1 \over n^2}\right) \right] \\
& \leq & e^{-\alpha/n} \rho(\sigma, \tau)
\eeas
This completes the proof.
\ep

\begin{thm}
\label{thm:mfBC}
Let $t_{\text{{\em \mix}}}(\ve)$ be the mixing time for the Glauber dynamics of the mean-field Blume-Capel model on $n$ vertices and $K_1(\beta)$ be the metastable critical point.  Then, for $\beta > \beta_c$ and $K < K_1(\beta)$,
\[ t_{\text{{\em \mix}}}(\ve) \leq {n \over \alpha} (\log n + \log (2/\ve)) \]
for any $\alpha \in \left(0,{K_1(\beta)-K \over K_1(\beta)}\right)$ and $n$ sufficiently large.
\end{thm}
\begin{proof}  Let $(X_t, Y_t)$ be a coupling of the Glauber dynamics of the BC model such that $Y_0 \overset{dist}{=} P_{n, \beta, K}$, the stationary distribution.  
For a given $\alpha \in \left(0,{K_1(\beta)-K \over K_1(\beta)}\right)$, let  $\varepsilon$ be as in Lemma \ref{lemma:aggregate}.  For sufficiently large $n$,
\beas
\| P^t(X_0, \cdot) - P_{n, \beta, K} \|_{TV} & \leq & P \{ X_t \neq Y_t \}  \\
& = & P \{ \rho(X_t, Y_t ) \geq 1 \} \\
& \leq & \mathbb{E} [ \rho(X_t,Y_t)] \\
& = & \mathbb{E} [ {\scriptstyle \mathbb{E}[\rho(X_t,Y_t)~|X_{t-1},Y_{t-1}]} ]\\
& \leq & \mathbb{E} [ {\scriptstyle \mathbb{E}[\rho(X_t,Y_t)~|X_{t-1},Y_{t-1}]}~|~|S_n(Y_{t-1})|<\varepsilon n] \cdot P\{|S_n(Y_{t-1})|<\varepsilon n\} \\
& & + nP\{|S_n(Y_{t-1})| \geq \varepsilon n\} 
\eeas

\noi
By iterating (\ref{eqn:contract}), it follows that 
\beas
\| P^t(X_0, \cdot) - P_{n, \beta, K} \|_{TV} 
& \leq & e^{-\alpha /n} \mathbb{E} [ \rho(X_{t-1},Y_{t-1})~|~|S_n(Y_{t-1})|<\varepsilon n] \cdot P\{|S_n(Y_{t-1})|<\varepsilon n\} \\
& & + nP\{|S_n(Y_{t-1})| \geq \varepsilon n\} \\
& \leq & e^{-\alpha /n} \mathbb{E} [ \rho(X_{t-1},Y_{t-1})]+nP\{|S_n(Y_{t-1})| \geq \varepsilon n\} \\
& \vdots & \vdots\\
& \leq & e^{-\alpha t/n} \mathbb{E} [ \rho(X_0,Y_0)] +n \sum_{s=0}^{t-1} P\{|S_n(Y_s)| \geq \varepsilon n\} \\
& = & e^{-\alpha t/n} \mathbb{E} [ \rho(X_0,Y_0)]+ntP_{n, \beta, K}\{|S_n/n| \geq \varepsilon\} \\
& \leq & n e^{-\alpha t/n}+ntP_{n, \beta, K}\{|S_n/n| \geq \varepsilon\} 
\eeas

\noi
We recall the result in Theorem \ref{thm:weakconv} that for $\beta > \beta_c$ and $K < K_1(\beta)$ 
\[ P_{n, \beta, K}\{ S_n/n \in dx \} \Longrightarrow \delta_0 \hsp \mbox{as $n \goto \infty$.} \]
Moreover, for any $\gamma>1$ and $n$ sufficiently large, the LDP stated in Theorem \ref{thm:ldppnbetak} implies that
\beas
\|P^t(X_0, \cdot) - P_{n, \beta, K} \|_{TV} & \leq & n e^{-\alpha t/n}+ntP_{n, \beta, K}\{|S_n/n| \geq \varepsilon\}  \\
&  <  & n e^{-\alpha t/n} +tne^{-{n \over \gamma}I_{\beta,K}(\varepsilon)} 
\eeas
For $t = {n \over \alpha} (\log n + \log (2/\ve))$, the above right-hand side converges to $\ve/2$ as $n \goto \infty$. 
\ep

\medskip

\subsection{Slow Mixing}  \label{slow}

In \cite{KOT}, the slow mixing region of the parameter space was determined for the mean-field Blume-Capel model.  Since the method used to prove the slow mixing, called the bottleneck ratio or Cheeger constant method, is not a coupling method, we simply state the result for completeness.

\begin{thm}
\label{thm:slowmix}
Let $t_{\text{{\em \mix}}} = t_{\text{{\em \mix}}}(1/4)$ be the mixing time for the Glauber dynamics of the mean-field Blume-Capel model on $n$ vertices.  For {\em (a)} $\beta \leq \beta_c$ and $K>K_c^{(2)}(\beta)$, and {\em (b)} $\beta > \beta_c$ and $K > K_1(\beta)$, there exists a positive constant $b$ and a strictly positive function $r(\beta,K)$ such that
\[ t_{\text{{\em \mix}}} \geq be^{r(\beta,K)n} \]
\end{thm}

\noi
We summarize the mixing time results for the mean-field Blume-Capel model and its relationship to the model's thermodynamic phase transition structure in Figure \ref{mixingphase}.  As shown in the figure, in the second-order, continuous phase transition region ($\beta \leq \beta_c$) for the BC model, the mixing time transition coincides with the equilibrium phase transition.  This is consistent with other models that exhibit this type of phase transition.  However, in the first-order, discontinuous phase transition region ($\beta > \beta_c$) the mixing time transition occurs below the equilibrium phase transition at the metastable critical value.

\begin{figure}[t]
\begin{center}
\includegraphics[height=3in]{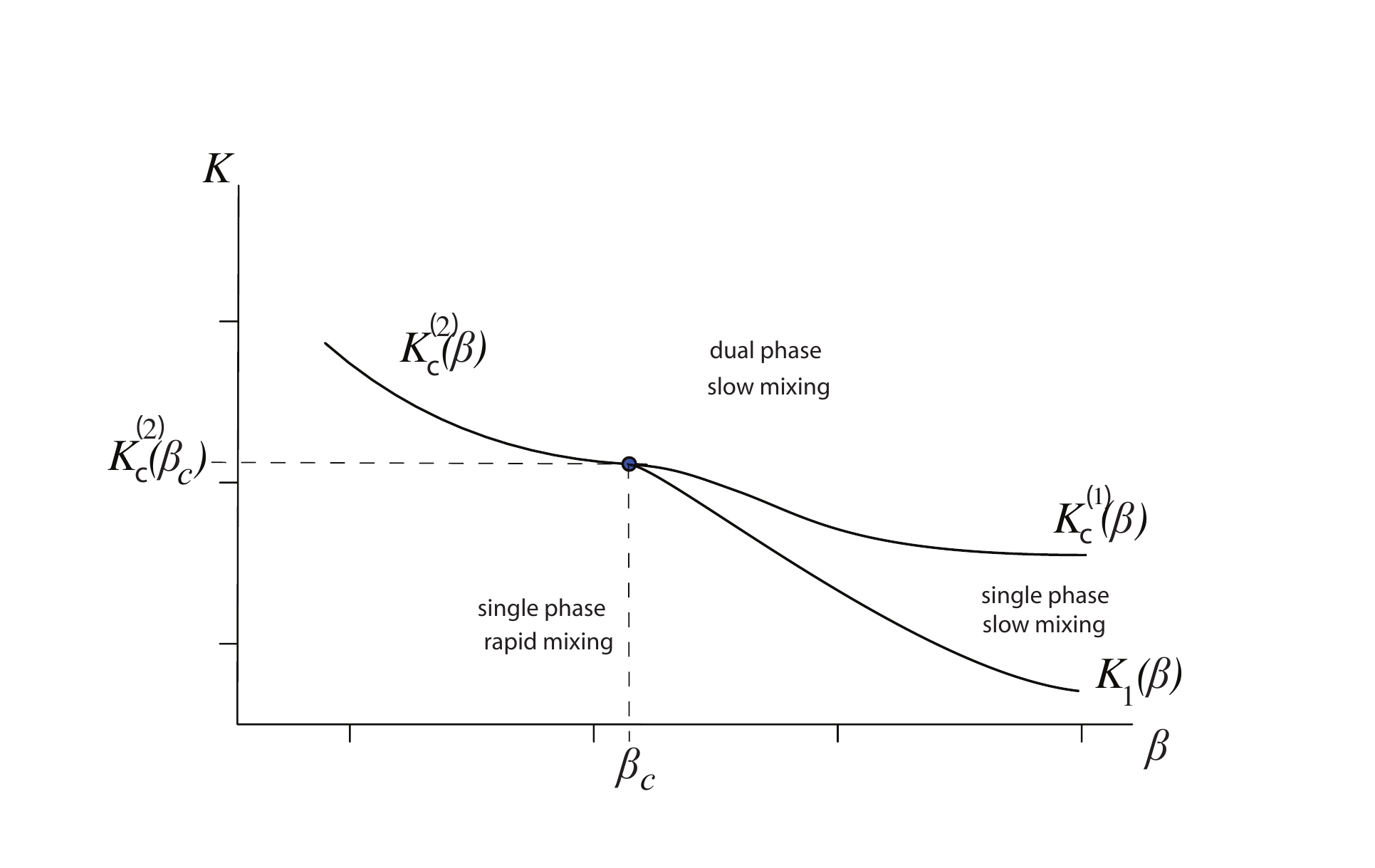}
\caption{\footnotesize Mixing times and equilibrium phase transition structure of the mean-field Blume-Capel model} \label{mixingphase}
\end{center}
\end{figure} 

\bigskip

\section{Aggregate Path Coupling for General Class of Gibbs Ensembles} \label{genspinmodels}

\medskip

In this section, we extend the aggregate path coupling technique derived in the previous section for the Blume-Capel model to a large class of statistical mechanical models that is disjoint from the mean-field Blume-Capel model.  The aggregate path coupling method presented here extends the classical path coupling method for Gibbs ensembles in two directions.  First, we consider macroscopic quantities in higher dimensions and find a monotone contraction path by considering a related variational problem in the continuous space. We also do not require the monotone path to be a nearest-neighbor path. In fact, in most situations we consider, a nearest-neighbor path will not work for proving contraction.  Second, the aggregation of the mean path distance along a monotone path  is shown to contract for some but not all pairs of configurations.  Yet, we use measure concentration and large deviation principle to show that showing contraction for pairs of configurations, where at least one of them is close enough to the equilibrium, is sufficient for establishing rapid mixing.

\medskip

Our main result is general enough to be applied to statistical mechanical models that undergo both types of phase transitions and to models whose macroscopic quantity are in higher dimensions.  Moreover, despite the generality, the application of our results requires straightforward conditions that we illustrate in Section \ref{GCWP}.  This is a significant simplification for proving rapid mixing for statistical mechanical models, especially those that undergo first-order, discontinuous phase transitions.  Lastly, our results also provide a link between measure concentration of the stationary distribution and rapid mixing of the corresponding dynamics for this class of statistical mechanical models.  This idea has been previously studied in \cite{MJLuczak} where the main result showed that rapid mixing implied measure concentration defined in terms of Lipschitz functions.  In our work, we prove a type of converse where measure concentration, in terms of a large deviation principle, implies rapid mixing.

\bigskip

\subsection{Class of Gibbs Ensembles} \label{sec:gibbs} 

\medskip

We begin by defining the general class of statistical mechanical spin models for which our results can be applied.  In Section \ref{GCWP}, we illustrate the application of our main result for the generalized Curie-Weiss-Potts model for which the mixing times has not been previously obtained.  

\medskip

Let $q$ be a fixed integer and define $\Lambda = \{ e^1, e^2, \ldots, e^q \}$, where $e^k$ are the $q$ standard basis vectors of $\R^q$.  A {\it configuration} of the model has the form $\omega = (\omega_1, \omega_2, \ldots, \omega_n) \in \Lambda^n$.  We will consider a configuration on a graph with $n$ vertices and let $X_i(\omega) = \omega_i$ be the {\it spin} at vertex $i$.  The random variables $X_i$'s for $i=1, 2, \ldots, n$ are independent and identically distributed with common distribution $\rho$.  

\medskip

In terms of the microscopic quantities, the spins at each vertex, the relevant macroscopic quantity is the {\it magnetization vector} (a.k.a empirical measure or proportion vector) 
\be
\label{eqn:empmeasure} 
L_n(\omega) = (L_{n,1}(\omega), L_{n,2}(\omega), \ldots, L_{n,q}(\omega)) 
\ee
where the $k$th component is defined by 
\[ L_{n, k}(\omega) = \frac{1}{n} \sum_{i=1}^n \delta(\omega_i, e^k) \]
which yields the proportion of spins in configuration $\omega$ that take on the value $e^k$.  The magnetization vector $L_n$ takes values in the set of probability vectors 
\be
\label{eqn:latticesimplex} 
\mathcal{P}_n = \left\{ \frac{n_k}{n} : \mbox{each} \ n_k \in \{0, 1, \ldots, n\} \ \mbox{and} \ \sum_{k=1}^q n_k = n \right\} 
\ee
inside the continuous simplex
$$ \mathcal{P} = \left\{ \nu \in \R^q : \nu = (\nu_1, \nu_2, \ldots, \nu_q), \mbox{each} \ \nu_k \geq 0, \ \sum_{k=1}^q \nu_k = 1 \right\}.$$

\medskip

\begin{remark}
For $q=2$, the empirical measure $L_n$ yields the empirical mean $S_n(\omega)/n$ where $S_n(\omega) = \sum_{i=1}^n \omega_i$.  Therefore, the class of models considered in this paper includes those where the relevant macroscopic quantity is the empirical mean, like the Curie-Weiss (mean-field Ising) model.
\end{remark}

As discussed in section \ref{gibbs}, statistical mechanical models are defined in terms of the Hamiltonian function, denoted by $H_n(\omega)$, which encodes the interactions of the individual spins and the total energy of a configuration. The link between the microscopic interactions to the macroscopic quantity, in this case $L_n(\omega)$, is the interaction representation function, which we define again for convenience below.

\begin{defn}
\label{defn:irf}
For $z \in \R^q$, we define the {\bf interaction representation function}, denoted by $H(z)$, to be a differentiable function satisfying 
\[ H_n (\omega) = n H(L_n(\omega)) \] 
\end{defn}

\medskip

\noi
Throughout the paper we suppose the interaction representation function $H(z)$ is a finite concave $\mathcal{C}^3(\R^q)$ function that has the form
$$H(z)=H_1(z_1)+H_2(z_2)+\ldots +H_q(z_q)$$
For example, for the Curie-Weiss-Potts (CWP) model \cite{CDLLPS},
$$H(z)=-{1 \over 2} \big<z,z \big>=-{1 \over 2} z_1^2-{1 \over 2} z_2^2-\ldots -{1 \over 2} z_q^2.$$

\medskip

The class Gibbs measures or Gibbs ensemble considered in this section is defined by
\be 
\label{eqn:gibbs}
P_{n, \beta} (B) = \frac{1}{Z_n(\beta)} \int_B \exp \left\{ -\beta H_n(\omega) \right\} dP_n = \frac{1}{Z_n(\beta)} \int_B \exp \left\{  -\beta n \, H\left(L_n(\omega) \right) \right\} dP_n 
\ee
where $P_n$ is the product measure with identical marginals $\rho$ and $Z_n(\beta) = \int_{\Lambda^n} \exp \left\{ -\beta H_n(\omega) \right\} dP_n$ is the {\bf partition function}.  The positive parameter $\beta$ represents the inverse temperature of the external heat bath.

\medskip

\begin{remark}
To simplify the presentation, we take $\Lambda = \{ e^1, e^2, \ldots, e^q \}$, where $e^k$ are the $q$ standard basis vectors of $\R^q$.  But our analysis has a straight-forward generalization to the case where $\Lambda = \{ \theta^1, \theta^2, \ldots, \theta^q \}$, where $\theta^k$ is any basis of $\R^q$.  In this case, the product measure $P_n$ would have identical one-dimensional marginals equal to 
\[ \bar{\rho} = \frac{1}{q} \sum_{i=1}^q \delta_{\theta^i} \]
\end{remark}

\medskip

An important tool we use to prove rapid mixing of the Glauber dynamics that converge to the Gibbs ensemble above is the large deviation principle of the empirical measure with respect to the Gibbs ensemble.  This measure concentration is precisely what drives the rapid mixing.  The large deviation principle for our class of Gibbs ensembles $P_{n, \beta}$ is presented next.

\bigskip

\subsection{Large Deviations} \label{ldp} 

\medskip

By Sanov's Theorem, the empirical measure $L_n$ satisfies the large deviation principle (LDP) with respect to the product measure $P_n$ with identical marginals $\rho$ and the rate function is given by the {\bf relative entropy}
\[ R(\nu | \rho) = \sum_{k=1}^q \nu_k \log \left( \frac{\nu_k}{\rho_k} \right) \]
for $\nu \in \mathcal{P}$.  Theorem 2.4 of \cite{EHT} yields the following result for the Gibbs measures (\ref{eqn:gibbs}).
\begin{thm}
\label{thm:sanov}
The empirical measure $L_n$ satisfies the LDP with respect to the Gibbs measure $P_{n, \beta}$ with rate function
\[ I_\beta(z) = R(z | \rho) + \beta H(z) - \inf_t \{ R(t | \rho) + \beta H(t) \}. \]
\end{thm}

\medskip

As discussed in section \ref{LDP}, the LDP upper bound stated in the previous theorem yields the following natural definition of {\bf equilibrium macrostates} of the model.
\be
\label{eqn:eqmac} 
\mathcal{E}_\beta := \left\{ \nu \in \mathcal{P} : \nu \ \mbox{minimizes} \ R(\nu | \rho) + \beta H(\nu) \right\} 
\ee
For our main result, we assume that there exists a positive interval $B$ such that for all $\beta \in B$, $\mathcal{E}_\beta$ consists of a single state $z_\beta$.  We refer to this interval $B$ as the single phase region.

\medskip

Again from the LDP upper bound, when $\beta$ lies in the single phase region, we get
\be
\label{eqn:ldplimit}
P_{n, \beta} (L_n \in dx) \Longrightarrow \delta_{z_\beta} \hsp \mbox{as} \hsp n \goto \infty. 
\ee
The above asymptotic behavior will play a key role in obtaining a rapid mixing time rate for the Glauber dynamics corresponding to the Gibbs measures (\ref{eqn:gibbs}).

\medskip

An important function in our work is the free energy functional defined below.  It is defined in terms of the interaction representation function $H$ and the logarithmic moment generating function of a single spin; specifically, for $z \in \R^q$ and $\rho$ equal to the uniform distribution, the {\bf logarithmic moment generating function} of $X_1$, the spin at vertex $1$, is defined by 
\be 
\label{eqn:lmgf}
\Gamma(z) = \log \left(  \frac{1}{q} \sum_{k=1}^q \exp\{z_k\}\right). 
\ee

\begin{defn}
\label{defn:fef}
The {\bf free energy functional} for the Gibbs ensemble $P_{n, \beta}$ is defined as
\be 
\label{eqn:fef}
G_{\beta} (z) = \beta (-H)^\ast (-\nabla H(z)) - \Gamma( -\beta \nabla H(z))
\ee
where for a finite, differentiable, convex function $F$ on $\R^q$, $F^\ast$ denotes its Legendre-Fenchel transform defined by 
\[ F^\ast(z) = \sup_{x \in \R^q} \{ \langle x, z \rangle - F(x) \} \]
\end{defn}

The following lemma yields an alternative formulation of the set of equilibrium macrostates of the Gibbs ensemble in terms of the free energy functional.  The proof is a straightforward generalization of Theorem A.1 in \cite{CET}.  

\begin{lemma}
\label{lemma:fefeqstates}
Suppose $H$ is finite, differentiable, and concave.  Then 
\[ \inf_{z \in \mathcal{P}} \{ R(z | \rho) + \beta H(z) \} = \inf_{z \in \R^q}  \{ G_\beta(z) \} \]
Moreover, $z_0 \in \mathcal{P}$ is a minimizer of $R(z | \rho) + \beta H(z)$ if and only if $z_0$ is a minimizer of $G_\beta(z)$.
\end{lemma}

\noi 
Therefore, the set of equilibrium macrostates can be expressed in terms of the free energy functional as 
\be
\label{eqn:fefeqst}
\mathcal{E}_\beta = \left\{ z \in \mathcal{P} : z \ \mbox{minimizes} \ G_\beta(z) \right\} 
\ee

As mentioned above, we consider only the single phase region of the Gibbs ensemble; i.e.\ values of $\beta$ where $G_\beta(z)$ has a unique global minimum.  For example, for the Curie-Weiss-Potts model \cite{CET}, the single phase region are values of $\beta$ such that $0 < \beta < \beta_c := (2(q-1)/(q-2)) \log(q-1)$.  At this critical value $\beta_c$, the model undergoes a first-order, discontinuous phase transition in which the single phase changes to a multiple phase discontinuously.  

\medskip

As we will show, the geometry of the free energy functional $G_\beta$ not only determines the equilibrium behavior of the Gibbs ensembles but it also yields the condition for rapid mixing of the corresponding Glauber dynamics.

\bigskip

\subsection{Glauber Dynamics} \label{sec:glauber} 

\medskip

On the configuration space $\Lambda^n$, we define the Glauber dynamics, defined in general in subsection \ref{mt}, for the class of Gibbs ensembles $P_{n, \beta}$ defined in (\ref{eqn:gibbs}).  These dynamics yield a reversible Markov chain $X^t$ with stationary distribution being the Gibbs ensemble $P_{n, \beta}$.

\medskip

(i) Select a vertex $i$ uniformly, 

\smallskip

(ii) Update the spin at vertex $i$ according to the distribution $P_{n, \beta}$, conditioned on the event that the spins at all vertices not equal to $i$ remain unchanged.  

\medskip

For a given configuration $\sigma = (\sigma_1, \sigma_2, \ldots, \sigma_n)$, denote by $\sigma_{i, e^k}$ the configuration that agrees with $\sigma$ at all vertices $j \neq i$ and the spin at the vertex $i$ is $e^k$; i.e. 
\[ \sigma_{i, e^k} = (\sigma_1, \sigma_2, \ldots, \sigma_{i-1}, e^k, \sigma_{i+1}, \ldots, \sigma_n) \]

Then if the current configuration is $\sigma$ and vertex $i$ is selected, the probability the spin at $i$ is updated to $e^k$, denoted by $P(\sigma \goto \sigma_{i, e^k})$, is equal to 
\be
\label{eqn:tranprob1} 
P(\sigma \goto \sigma_{i, e^k}) = \frac{\exp\big\{-\beta n H(L_n(\sigma_{i, e^k}))\big\}}{ \sum_{\ell=1}^q \exp\big\{-\beta n H(L_n(\sigma_{i, e^\ell}))\big\}}. 
\ee

\medskip

Next, we show that the update probabilities of the Glauber dynamics above can be expressed in terms of the derivative of the logarithmic moment generating function of the individual spins $\Gamma$ defined in (\ref{eqn:lmgf}). The partial derivative of $\Gamma$ in the direction of $e^\ell$ has the form 
\[ \left[\partial_{\ell} \Gamma \right] (z) = \frac{\exp\{z_\ell\}}{\sum_{k=1}^q \exp\{z_k\}} \]

\medskip

\noi
We introduce the following function that plays the key role in our analysis.
\be
\label{eqn:gell}
g_{\ell}^{H, \beta}(z) = \left[ \partial_{\ell} \Gamma \right] (-\beta \nabla H(z)) = \frac{\exp\left( -\beta \, [\partial_\ell H](z) \right)}{ \sum_{k=1}^q \exp \left( -\beta \, [\partial_k H](z)  \right) }.
\ee
Denote
\be
\label{eqn:littleG}
g^{H, \beta}(z):=\Big(g_1^{H, \beta}(z),\ldots,g_q^{H, \beta}(z) \Big).
\ee
Note that $g^{H, \beta}(z)$ maps the simplex
$$ \mathcal{P} = \left\{ \nu \in \R^q : \nu = (\nu_1, \nu_2, \ldots, \nu_q), \mbox{each} \ \nu_k \geq 0, \ \sum_{k=1}^q \nu_k = 1 \right\} $$
into itself and it can be expressed in terms of the free energy functional $G_\beta$ defined in (\ref{eqn:fef}) by 
\[ \nabla G_\beta(z) = \beta [ \nabla (-H)^\ast (-\nabla H(z)) - g^{H, \beta}(z) ] \]

\medskip

\begin{lemma}
\label{lemma:transprob}
Let $P(\sigma \goto \sigma_{i, e^k})$ be the Glauber dynamics update probabilities given in (\ref{eqn:tranprob1}).  
Then, for any $k \in \{1,2,\ldots, q \}$,
$$P(\sigma \goto \sigma_{i, e^k}) =  \left[ \partial_k \Gamma \right] \Big(-\beta \nabla H(L_n(\sigma))-{\beta \over 2n}\mathcal{Q}H(L_n(\sigma))+{\beta \over n} \Big< \sigma_i, \mathcal{Q}H(L_n(\sigma))\Big> \sigma_i \Big) + O\left(\frac{1}{n^2} \right),$$
where $\mathcal{Q}$ is the following linear operator:
$$\mathcal{Q} F(z):=\left(\partial_1^2 F(z), ~\partial_2^2 F(z), ~\ldots , ~\partial_q^2 F(z)\right),$$
for any $~F: \mathbb{R}^q \rightarrow \mathbb{R}~$ in $\mathcal{C}^2$.
\end{lemma}
\bp
Suppose $\sigma_i=e^m$. By Taylor's theorem, for any $k \not= m$, we have 
\beas
H(L_n(\sigma_{i, e^k})) & = & H(L_n(\sigma)) +H_m\big(L_{n, m} (\sigma) -1/n \big)-H_m\big(L_{n, m} (\sigma) \big) \\
&  & + ~H_k\big(L_{n, k} (\sigma) +1/n \big)-H_k\big(L_{n, k} (\sigma) \big) \\
& = & H(L_n(\sigma)) + \frac{1}{n} \left[ \partial_{k} H(L_n(\sigma)) - \partial_m H(L_n(\sigma)) \right]\\
& & +~{1 \over 2n^2} \left[ \partial^2_{k} H(L_n(\sigma)) + \partial^2_m H(L_n(\sigma)) \right]+ O\left(\frac{1}{n^3} \right). 
\eeas
Now, if $k=m$,
\beas
H(L_n(\sigma_{i, e^k}))  & = & H(L_n(\sigma))\\
& = & H(L_n(\sigma)) + \frac{1}{n} \left[ \partial_{k} H(L_n(\sigma)) -\partial_m H(L_n(\sigma)) \right] \\ 
& & \hsp +{1 \over 2n^2} \left[ -\partial^2_{k} H(L_n(\sigma)) +\partial^2_m H(L_n(\sigma)) \right]. 
\eeas
This implies that the transition probability (\ref{eqn:tranprob1}) has the form
\[
P(\sigma \goto \sigma_{i, e^k}) = \left[ \partial_k \Gamma \right] \Big(-\beta \nabla H(L_n(\sigma))-{\beta \over 2n}\mathcal{Q}H(L_n(\sigma))+{\beta \over n}\partial^2_{m}H(L_n(\sigma))e^m \Big) + O\left(\frac{1}{n^2} \right) 
\]
as $\exp\big\{O\left(\frac{1}{n^2} \right) \big\}=1+O\left(\frac{1}{n^2} \right)$.
\ep

\medskip

\noindent
The above Lemma \ref{lemma:transprob} can be restated as follows using Taylor expansion.
\begin{cor} \label{cor:tp}
Let $P(\sigma \goto \sigma_{i, e^k})$ be the Glauber dynamics update probabilities given in (\ref{eqn:tranprob1}).  
Then, for any $k \in \{1,2,\ldots, q \}$,
$$P(\sigma \goto \sigma_{i, e^k}) = g_k^{H, \beta}(L_n(\sigma)) +{ \beta \over n} \varphi_{k, \sigma_i}^{H, \beta}(L_n(\sigma))+O\left(\frac{1}{n^2} \right),$$
where
$$\varphi_{k, e^r}^{H, \beta}(z):=-{1 \over 2}\Big<\mathcal{Q}H(z),  \left[\nabla \partial_{\ell} \Gamma \right] (-\beta \nabla H(z)) \Big>+\Big< e^r, \mathcal{Q} H(z)\Big> \Big<e^r,  \left[\nabla \partial_{\ell} \Gamma \right] (-\beta \nabla H(z)) \Big>  .$$
\end{cor}

\medskip

In the next section, we define the specific coupling used to bound the mixing time of the Glauber dynamics for the class of Gibbs ensembles defined in subsection \ref{sec:gibbs}.
 
\bigskip

\subsection{Coupling of Glauber Dynamics} \label{glauber}

\medskip

We begin by defining a metric on the configuration space $\Lambda^n$.  For two configurations $\sigma$ and $\tau$ in $\Lambda^n$, define
\be
\label{eqn:pathmetric} 
d(\sigma, \tau) = \sum_{j=1}^n 1\{ \sigma_j \neq \tau_j \} 
\ee
which yields the number of vertices at which the two configurations differ.

\medskip

Let $X^t$ and $Y^t$ be two copies of the Glauber dynamics. Here, we use the standard greedy coupling of $X^t$ and $Y^t$.  At each time step a vertex is selected at random, uniformly from the $n$ vertices.     Suppose $X^t=\sigma$, $~Y^t=\tau$, and the vertex selected is denoted by $j$. Next, we erase the spin at location $j$ in both processes, and replace it with a new one according to the following update probabilities. For all $\ell = 1, 2, \ldots, q$, define
\[ p_\ell = P(\sigma \goto \sigma_{j, e^\ell}) \quad \mbox{and} \quad q_\ell = P(\tau \goto \tau_{j, e^\ell}) \]
and let
\[ P_\ell = \min\{ p_\ell, q_\ell \} \hsp \mbox{and} \hsp P = \sum_{\ell = 1}^q P_\ell . \]
Now, let $B$ be a Bernoulli random variable with probability of success $P$.  If $B = 1$, we update the two chains equally with the following probabilities
\[ P(X_j^{t+1} = e^\ell, Y_j^{t+1} = e^\ell \, | \, B=1) = \frac{P_\ell}{P} \]
for $\ell = 1, 2, \ldots, q$.  On the other hand, if $B= 0$, we update the chains differently according to the following probabilities
\[ P(X_j^{t+1} = e^\ell, Y_j^{t+1} = e^m \, | \, B=0) = \frac{p_\ell - P_\ell}{1-P} \cdot \frac{q_m - P_m}{1-P} \] 
for all pairs $\ell \neq m$.  Then the total probability that the two chains update the same is equal to $P$ and
the total probability that the chains update differently is equal to $1-P$.

\medskip

\noi
Observe that once $X^t=Y^t$, the processes remain matched (coupled) for the rest of the time. In the coupling literature, the time 
$$\min\{t \geq 0 ~:~ X^t=Y^t\}$$
is refered to as the {\it coupling time}.

\medskip

As discussed in section \ref{mixtimepathcoupling}, the mean coupling distance  $\mathbb{E}[d(X^t, Y^t) ]$ is tied to the total variation distance via the following inequality: 
\be \label{coupling_ineq} 
\| P^t(x, \cdot) - P^t(y, \cdot) \|_{{\scriptsize \mbox{TV}}} \leq P(X^t \neq Y^t)  \leq \mathbb{E}[d(X^t, Y^t) ] 
\ee
The above inequality implies that the order of the mean coupling time is an upper bound on the order of the mixing time. See \cite{L} and \cite{LPW} for details on coupling and coupling inequalities.

\medskip

\subsection{Mean Coupling Distance} \label{sec:mean} 

\medskip
Fix $~\varepsilon>0$. Consider two configurations $\sigma$ and $\tau$ such that
$$d(\sigma,\tau)=d,$$
where $d(\sigma, \tau) \in \mathbb{N}$ is the metric defined in (\ref{eqn:pathmetric}) and $\ve \leq \|L_n(\sigma)-L_n(\tau)\|_1 < 2\varepsilon$.

\medskip
Let $\mathcal{I}=\{i_1,\hdots,i_d\}$ be the set of vertices at which the spin values of the two configurations $\sigma$ and $\tau$ disagree.  Define $\kappa(e^\ell)$ to be the probability that the coupled processes update differently when the chosen vertex $j\not\in \mathcal{I}$ has spin $e^\ell$. If the chosen vertex $j$ is such that $\sigma_j=\tau_j=e^\ell$, then by Corollary \ref{cor:tp} of  Lemma \ref{lemma:transprob},

\bea\label{eqn:true} \nonumber
 \kappa(e^\ell) & :=  & {1 \over 2}\sum_{k=1}^q \Big|P(\sigma \goto \sigma_{j, e^k}) -P(\tau \goto \tau_{j, e^k}) \Big|\\ \nonumber
 & & \\ \nonumber
& =  & {1 \over 2}\sum_{k=1}^q \Big|  \Big(g_k^{H, \beta}(L_n(\sigma)) +{ \beta \over n} \varphi_{k, e^\ell}^{H, \beta}(L_n(\sigma)) \Big)
- \Big(g_k^{H, \beta}(L_n(\tau)) +{ \beta \over n} \varphi_{k, e^\ell}^{H, \beta}(L_n(\tau)) \Big)  \Big| +O\left( \frac{1}{n^2} \right)\\ 
 & & \\ \nonumber
& =  & {1 \over 2}\sum_{k=1}^q \Big|  g_k^{H, \beta}(L_n(\sigma)) - g_k^{H, \beta}(L_n(\tau))   \Big| +O\left(\frac{\varepsilon}{n}+ \frac{1}{n^2} \right). \nonumber
\eea

\medskip
\noindent
Next, we observe that for any $\mathcal{C}^2$ function $f:~\mathcal{P} \rightarrow \mathbb{R}$, there exists $C>0$ such that
\be
\left| f(z') - f(z) - \Big<z' - z, \nabla f(z) \Big>~\right| < C\varepsilon^2
\ee
for all $z, z' \in \mathcal{P}$ satisfying $\varepsilon \leq \|z'-z\|_1 < 2\varepsilon$.

\bigskip
\noindent
Therefore for $n$ large enough, there exists $C'>0$ such that
\begin{equation}\label{eqn:rho}
\left| \kappa(e^\ell) ~-~ {1 \over 2}\sum_{k=1}^q \Big| \Big<L_n(\tau) - L_n(\sigma), \nabla g_k^{H, \beta}(L_n(\sigma)) \Big> \Big| \right| ~<C'\varepsilon^2.
\end{equation}
The above result holds regardless of the value of $\ell \in \{1,2,\dots,q\}$. 

\medskip
\noindent
Similarly, when the chosen vertex $j \in \mathcal{I}$, the probability of not coupling at $j$ satisfies (\ref{eqn:rho}).

\bigskip
\noindent
We conclude that in terms of $\kappa_{\sigma,\tau} := {1 \over 2}\sum_{k=1}^q \Big| \Big<L_n(\tau) - L_n(\sigma), \nabla g_k^{H, \beta}(L_n(\sigma)) \Big> \Big|$, the mean distance between a coupling of the Glauber dynamics starting in $\sigma$ and $\tau$ with $d(\sigma, \tau) = d$ after one step has the form
\begin{eqnarray} \label{eqn:grad}
\mathbb{E}_{\sigma,\tau}[d(X,Y)] & \leq  & d-{d \over n} (1-\kappa_{\sigma,\tau})+{n-d \over n}\kappa_{\sigma,\tau} +c \, \varepsilon^2 \nonumber \\
&  = & d \cdot \left[1-{1 \over n}\left(1-{\kappa_{\sigma,\tau}+c \, \varepsilon^2  \over d/n} \right)  \right]
\end{eqnarray} 
for a fixed $c>0$ and all $\ve$ small enough.

\medskip

\subsection{Aggregate Path Coupling} \label{sec:agg}

\medskip

In the previous section, we derived the form of the mean distance between a coupling of the Glauber dynamics starting in two configurations whose distance is bounded.  We next derive the form of the mean coupling distance of a coupling starting in two configurations that are connected by a path of configurations where the distance between successive configurations are bounded.   

\begin{defn}
Let $\sigma$ and $\tau$ be configurations in $\Lambda^n$. We say that a path $\pi$ connecting configurations $\sigma$ and $\tau$ denoted by 
$$\pi: \ \sigma = x_0, x_1, \ldots, x_r = \tau ,$$
is a {\bf monotone path} if
\begin{itemize}
\item[(i)] $~\sum\limits_{i=1}^r d(x_{i-1},x_i) = d(\sigma,\tau)$

\item[(ii)] for each $k=1, 2, \ldots, q$, the $k$th coordinate of $L_n(x_i)$, $L_{n,k}(x_i)$ is monotonic as $i$ increases from $0$ to $r$;\\
\end{itemize}
\end{defn}
\noi
Observe that here the points $x_i$ on the path are not required to be nearest-neighbors.

\medskip

\noi
A straightforward property of monotone paths is that 
\[ \sum_{i=1}^r \sum_{k=1}^q L_{n,k}(x_i) - L_{n,k}(x_{i-1}) = L_n(\sigma) - L_n(\tau)   \]
Another straightforward observation is that for any given path
$$L_n(\sigma)=z_0, z_1,\hdots,z_r=L_n(\tau)$$
in $\mathcal{P}_n$, monotone in each coordinate, with $~\|z_i-z_{i-1}\|_1>0$ for all $i \in \{1,2,\hdots,r\}$, there exists a monotone path
$$\pi: \ \sigma = x_0, x_1, \ldots, x_r = \tau$$
such that $L_n(x_i)=z_i$ for each $i$.

\medskip

Let $\pi : \sigma = x_0, x_1, \hdots, x_r = \tau$ \ be a monotone path connecting configurations $\sigma$ and $\tau$ such that $\varepsilon \leq \|L_n(x_i) - L_n(x_{i-1}) \|_1 < 2\varepsilon $ for all $i = 1, \ldots, r$.   Equation (\ref{eqn:grad}) implies the following bound on the mean distance between a coupling of the Glauber dynamics starting in configurations $\sigma$ and $\tau$:

{\small
\bea \label{eqn:meandist}
\lefteqn{ \mathbb{E}_{\sigma, \tau}[d(X,Y)] } \nonumber \\
& \leq & \sum_{i=1}^r \mathbb{E}_{x_{i-1}, x_i}[d(X_{i-1}, X_i)]   \nonumber \\
 & \leq &  \sum\limits_{i=1}^r \left\{ d(x_{i-1},x_i) \cdot \left[1 -{1 \over n} \left(1-{{1 \over 2} \sum\limits_{k=1}^q \Big| \Big<L_n(x_i) - L_n(x_{i-1}), \nabla g_k^{H, \beta}(L_n( x_{i-1})) \Big> \Big|+c \varepsilon^2 \over d(x_{i-1},x_i)/n} \right)  \right]\right\}  \nonumber \\ 
& = & d(\sigma,\tau) \left[ 1-{1 \over n}\left(1-{\sum\limits_{k=1}^q  \sum\limits_{i=1}^r \Big| \Big<L_n(x_i) - L_n(x_{i-1}), \nabla g_k^{H, \beta}(L_n( x_{i-1})) \Big> \Big|+c \varepsilon^2 \over  2d(\sigma,\tau)/n }\right) \right]  \nonumber \\ 
 & & \\
& \leq & d(\sigma,\tau) \left[ 1-{1 \over n}\left(1-{\sum\limits_{k=1}^q  \sum\limits_{i=1}^r \Big| \Big<L_n(x_i) - L_n(x_{i-1}), \nabla g_k^{H, \beta}(L_n( x_{i-1})) \Big> \Big|+c \varepsilon^2 \over  \|L_n(\sigma)-L_n(\tau)\|_1 }\right) \right], \nonumber
\eea
}

\noindent
as $~\sum\limits_{i=1}^r d(x_{i-1},x_i) = d(\sigma,\tau)$.

\medskip

From inequality (\ref{eqn:meandist}), if there exists monotone paths between all pairs of configurations such that there is a uniform bound less than $1$ on the ratio 
\[ {\sum_{k=1}^q  \sum_{i=1}^r \Big| \Big<L_n(x_i) - L_n(x_{i-1}), \nabla g_k^{H, \beta}(L_n( x_{i-1})) \Big> \Big| \over  \|L_n(\sigma)-L_n(\tau)\|_1 } \]
then the mean coupling distance contracts which yields a bound on the mixing time via coupling inequality (\ref{coupling_ineq}).

\medskip

Although the Gibbs measure are distributions of the empirical measure $L_n$ defined on the discrete space $\mathcal{P}_n$, proving contraction of the mean coupling distance is often facilitated by working in the continuous space, namely the simplex $\mathcal{P}$.  We begin our discussion of aggregate path coupling by defining distances along paths in $\mathcal{P}$.  

\medskip

Recall the function $g^{H,\beta}$ defined in (\ref{eqn:littleG}) which is dependent on the Hamiltonian of the model through the interaction representation function $H$ defined in Definition \ref{defn:irf}. 

\begin{defn}
Define the {\bf aggregate $g$-variation} between a pair of points $x$ and $z$ in $\mathcal{P}$ along a continuous monotone  (in each coordinate) path $\rho$ to be 
$$ D_\rho^g (x, z) := \sum\limits_{k=1}^q\int\limits_{\rho} \Big| \Big<\nabla g_k^{H, \beta}(y), dy \Big> \Big| $$
Define the corresponding {\bf pseudo-distance} between a pair of points points $x$ and $z$ in $\mathcal{P}$ as
$$d_g(x,z) :=\inf_{\rho}D_\rho^g (x, z),$$
where the infimum is taken over all continuous monotone paths in $\mathcal{P}$ connecting $x$ and $z$. 
\end{defn}

Notice if the monotonicity restriction is removed, the above infimum would satisfy the triangle inequality.
We will need the following condition.
\begin{cond} \label{uniform}
Let $z_\beta$ be the unique equilibrium macrostate. There exists $\delta \in (0,1)$ such that
$${d_g(z,z_\beta) \over \|z-z_\beta\|_1} \leq 1-\delta$$
for all $z$ in $\mathcal{P}$.
\end{cond}
\noindent
Observe that  if it is shown that $~d_g(z,z_\beta) < \|z-z_\beta\|_1$ for all $z$ in $\mathcal{P}$, then by continuity the above condition is equivalent to
$$\limsup_{z \rightarrow z_\beta} {d_g(z,z_\beta) \over \|z-z_\beta\|_1} <1$$

\medskip
\noindent
Suppose Condition \ref{uniform} is satisfied. Then let denote by ${\bf NG}_\delta$ the family of  smooth curves, monotone in each coordinate such that for each $z \not= z_\beta$ in $\mathcal{P}$, there is exactly one curve $\rho=\rho_z$ in the family  ${\bf NG}_\delta$ connecting $z_\beta$ to $z$, and
$${D_\rho^g (z, z_\beta) \over \|z-z_\beta\|_1} \leq 1-\delta/2.$$
Such family of smooth curves will be referred to as {\bf neo-geodesic}.
\medskip
\noindent
\begin{cond}\label{RS}
For $\varepsilon>0$ small enough, there exists a neo-geodesic family ${\bf NG}_\delta$ such that for each $z$ in $\mathcal{P}$ satisfying $\|z-z_\beta\|_1 \geq \varepsilon$ , the curve  $\rho=\rho_z$ in the family  ${\bf NG}_\delta$ that connects $z_\beta$ to $z$ satisfies
$${\sum_{k=1}^q  \sum_{i=1}^r \Big| \Big<z_i - z_{i-1}, \nabla g_k^{H, \beta}(z_{i-1}) \Big> \Big| \over \|z-z_\beta\|_1} \leq 1-\delta/3$$
for a sequence of points $z_0=z_\beta,z_1,\hdots,z_r=z$ interpolating $\rho$ such that
$$\varepsilon \leq \|z_i - z_{i-1}\|_1  < 2\varepsilon \quad \text{ for } i=1,2,\hdots, r.$$
\end{cond}

\medskip
\noindent
It is important to observe that Condition \ref{uniform} is often simpler to verify than Condition \ref{RS}. Moreover,  under certain simple additional prerequisites, Condition \ref{uniform} implies Condition \ref{RS}. For example, this is achieved if there is a uniform bound on the Cauchy curvature at every point of every curve in ${\bf NG}_\delta$. So it will be demonstrated on the example of the generalized Curie-Weiss-Potts model in section \ref{GCWP} that  the natural way for  establishing  Condition \ref{RS} for the model is via first establishing Condition \ref{uniform}.

\medskip

In addition to Condition \ref{RS} that will be shown to imply contraction when one of the two configurations in the coupled processes is at the equilibrium, i.e. $L_n(\sigma)=z_\beta$ , we need a condition that will imply contraction between two configurations within a neighborhood of the equilibrium configuration.  We state this assumption next.  

\begin{cond}
\label{ass:localbound}
Let $z_\beta$ be the unique equilibrium macrostate. Then,
$$\limsup\limits_{z \rightarrow z_\beta} {\|g^{H, \beta}(z)-g^{H, \beta}(z_\beta)\|_1 \over \|z-z_\beta\|_1} <1.$$
\end{cond}

\noindent
Since $H(z) \in \mathcal{C}^3$, the above Condition \ref{ass:localbound} implies that for any $\ve>0$ sufficiently small, there exists $\gamma \in (0,1)$ such that
$$ {\|g^{H, \beta}(z)-g^{H, \beta}(w)\|_1 \over \|z-w\|_1} <1-\gamma$$
for all $z$ and $w$ in $\mathcal{P}$ satisfying
$$\|z-z_\beta\|_1<\ve \quad \text{ and }\quad \|w-z_\beta\|_1<\ve.$$

\bigskip

\subsection{Main Result} \label{sec:main} 

\medskip

As discussed in section \ref{mixtimepathcoupling}, a sufficient condition for rapid mixing of the Glauber dynamics of Gibbs ensembles is contraction of the mean coupling distance $\mathbb{E}_{\sigma, \tau}[d(X,Y)]$ between coupled processes starting in all pairs of configurations in $\Lambda^n$.  The classical path coupling argument stated in Proposition \ref{classicalPC} is a method of obtaining this contraction by only proving contraction between couplings starting in \underline{neighboring} configurations.  However for some classes of models (e.g. models that undergo a first-order, discontinuous phase transition) there are situations when Glauber dynamics exhibits rapid mixing, but coupled processes do not exhibit contraction between some neighboring configurations.  A major strength of the {\bf aggregate path coupling} method is that it yields a proof for rapid mixing even in those cases when contraction of the mean distance between couplings starting in all pairs of neighboring configurations does not hold.

\medskip

The strategy is to take advantage of the large deviations estimates discussed in section \ref{ldp}.  Recall from that section that we assume that the set of equilibrium macrostates $\mathcal{E}_\beta$, which can be expressed in the form given in (\ref{eqn:fefeqst}), consists of a single point $z_\beta$.  Define an {\it equilibrium configuration} $\sigma_\beta$ to be a configuration such that 
\[ L_n(\sigma_\beta) = z_\beta = ((z_\beta)_1, (z_\beta)_2, \ldots, (z_\beta)_q). \]  
First we observe that in order to use the coupling inequality (\ref{coupling_ineq}) we need to show contraction of the mean coupling distance $\mathbb{E}_{\sigma, \tau}[d(X,Y)]$ between a  Markov chain initially distributed according to the stationary probability distribution $ P_{n, \beta}$ and a Markov chain starting at any given configuration.
Using large deviations we know that with high probability the former process starts near the equilibrium and stays near the equilibrium for long duration of time.

Our main result Theorem \ref{thm:main}  states that once we establish contraction of the mean coupling distance between two copies of a Markov chain where one of the coupled dynamics starts near an equilibrium configuration in Lemma \ref{lemma:contract}, then this contraction, along with the large deviations estimates of the empirical measure $L_n$, yields rapid mixing of the Glauber dynamics converging to the Gibbs measure. 

\medskip

Now, the classical path coupling relies on showing contraction along any monotone path connecting two configurations, in one time step. Here we observe that we only need to show contraction along \underline{one} monotone path connecting two configurations in order to have the mean coupling distance $\mathbb{E}_{\sigma, \tau}[d(X,Y)]$ contract in a single time step. However, finding even one monotone path with which we can show contraction in the equation (\ref{eqn:meandist}) is not easy. The answer to this is in finding a monotone path $\rho$ in $\mathcal{P}$ connecting the $L_n$ values of the two configurations, $\sigma$ and $\tau$, such that 
$${\sum\limits_{k=1}^q\int\limits_{\rho} \Big| \Big<\nabla g_k^{H, \beta}(y), dy \Big> \Big| \over  \|L_n(\sigma)-L_n(\tau)\|_1 } ~<1 $$
Although $\rho$ is a continuous  path in continuous space $\mathcal{P}$, it serves as Ariadne's thread for finding a monotone path
$$\pi:~\sigma=x_0,~x_1,\hdots,x_r=\tau$$
such that $L_n(x_0),~L_n(x_1),\hdots,L_n(x_r)$ in $\mathcal{P}_n$ are positioned along $\rho$, and
$$\sum\limits_{k=1}^q  \sum\limits_{i=1}^r \Big| \Big<L_n(x_i) - L_n(x_{i-1}), \nabla g_k^{H, \beta}(L_n( x_{i-1})) \Big> \Big|$$
is a Riemann sum approximating $~\sum\limits_{k=1}^q\int\limits_{\rho} \Big| \Big<\nabla g_k^{H, \beta}(y), dy \Big> \Big|$. Therefore we obtain
$${\sum\limits_{k=1}^q  \sum\limits_{i=1}^r \Big| \Big<L_n(x_i) - L_n(x_{i-1}), \nabla g_k^{H, \beta}(L_n( x_{i-1})) \Big> \Big| \over  \|L_n(\sigma)-L_n(\tau)\|_1 } ~<1,$$
that in turn implies contraction in (\ref{eqn:meandist}) for $\ve$ small enough and $n$ large enough. See Figure \ref{fig:simplex}.

\begin{figure} 
\includegraphics[scale=0.4]{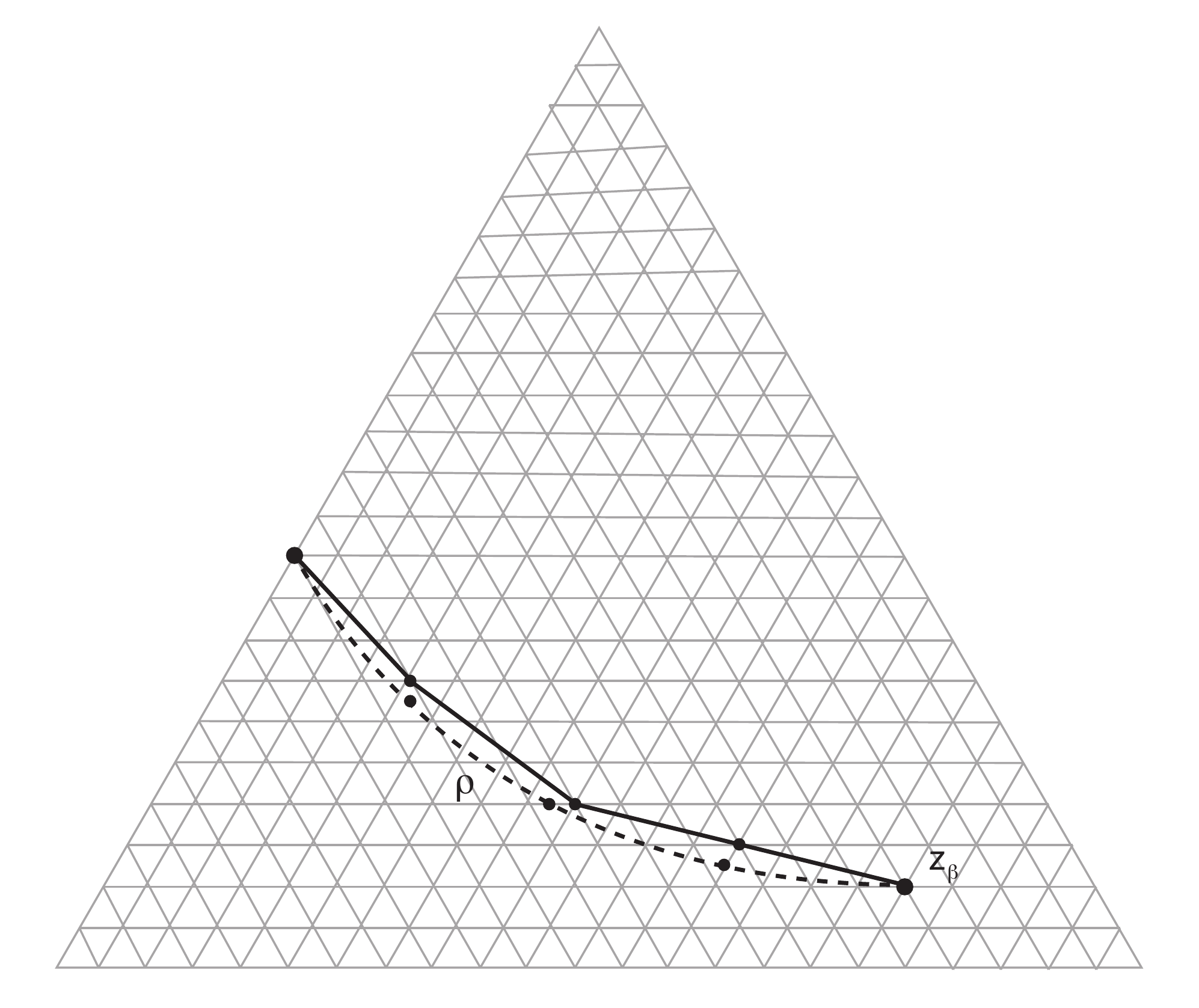}
\caption{Case $q=3$. Dashed curve is the continuous monotone path $\rho$. Solid lines represent the path $L_n(x_0),~L_n(x_1),\hdots,L_n(x_r)$ in $\mathcal{P}_n$.}
\label{fig:simplex}
\end{figure}

Observe that in order for the above argument to work, we need to spread points $L_n(x_i) \in \mathcal{P}_n$ along a continuous path $\rho$ at intervals of fixed order $\ve$.  Thus $\pi$ has to be {\bf not a nearest-neighbor path} in the space of configurations, another significant deviation from the classical path coupling.

\medskip

\begin{lemma}
\label{lemma:contract}
Assume Condition \ref{RS} and Condition \ref{ass:localbound}. Let $(X,Y)$ be a coupling of the Glauber dynamics as defined in Section \ref{glauber}, starting in configurations $\sigma$ and $\tau$ and let $z_\beta$ be the single equilibrium macrostate of the corresponding Gibbs ensemble. 
Then there exists an $\alpha>0$ and an $\varepsilon'>0$ small enough such that for $n$ large enough,
\[ \mathbb{E}_{\sigma,\tau} [ d(X,Y)] \leq e^{-\alpha/n} d(\sigma, \tau) \]
whenever $\|L_n(\sigma) -z_\beta \|_1<\varepsilon' $.
\end{lemma}
\bp
Let $\ve$ and $\delta$ be as in Condition \ref{RS}, and let $\ve'=\ve^2 \delta/M$ with a constant $M\gg0$. 

\medskip
\noindent
{\bf Case I.} Suppose $L_n(\tau)=z$ and $L_n(\sigma)=w$, where $\|z-z_\beta\|_1 \geq \ve$ and $\|w-z_\beta\|_1< \ve'$.

\medskip
\noindent
Then there is an equlibrium configuration $\sigma_\beta$ with $L_n(\sigma_\beta)=z_\beta$ such that there is a {\it monotone path}
$$\pi':~\sigma_\beta=x'_0,~x'_1,\hdots,x'_r=\tau$$
connecting configurations $\sigma_\beta$ and $\tau$ on $\Lambda^n$ such that
$\ve  \leq\|L_n(x'_i) - L_n(x'_{i-1})\|_1<2\ve$, and by Condition \ref{RS}, 
$${\sum\limits_{k=1}^q  \sum\limits_{i=1}^r \Big| \Big<L_n(x'_i) - L_n(x'_{i-1}), \nabla g_k^{H, \beta}(L_n( x'_{i-1})) \Big> \Big| \over  \|L_n(\sigma_\beta)-L_n(\tau)\|_1 } \leq 1-\delta/4$$
for $n$ large enough. Note that the difference between the above inequality and Condition \ref{RS} is that here we take $L_n(x'_i) \in \mathcal{P}_n$.

\medskip
\noindent
Now,  there exists a monotone path with from $\sigma$ to $\tau$
$$\pi:~\sigma=x_0,~x_1,\hdots,x_r=\tau$$
such that 
$$\|L_n(x_i)-L_n(x'_i)\|_1 \leq \ve' \qquad \text{ for all }i=0,1,\hdots,r.$$
The new monotone path $\pi$ is constructed from $\pi'$ by insuring that either
$$0 \leq \Big<L_n(x_i)-L_n(x_{i-1}),e^k\Big> \leq  \Big<L_n(x'_i)-L_n(x'_{i-1}),e^k\Big> $$
or
$$ \Big<L_n(x'_i)-L_n(x'_{i-1}),e^k\Big>  \leq  \Big<L_n(x_i)-L_n(x_{i-1}),e^k\Big>  \leq 0$$
for $i=2,\hdots,r$ and each coordinate $k \in \{1,2,\hdots,q\}$.

\medskip
\noindent
Then
{\footnotesize
$$\left|{\sum\limits_{k=1}^q  \sum\limits_{i=1}^r \Big| \Big<L_n(x_i) - L_n(x_{i-1}), \nabla g_k^{H, \beta}(L_n( x_{i-1})) \Big> \Big| \over  \|L_n(\sigma)-L_n(\tau)\|_1 }
-{\sum\limits_{k=1}^q  \sum\limits_{i=1}^r \Big| \Big<L_n(x'_i) - L_n(x'_{i-1}), \nabla g_k^{H, \beta}(L_n( x'_{i-1})) \Big> \Big| \over  \|L_n(\sigma_\beta)-L_n(\tau)\|_1 }\right| \leq C'' r \ve' / \ve$$
}
\noi
for a fixed constant $C''>0$. Noticing that $~r\ve'/\ve \leq \delta /M~$ as $~r \leq 1/\ve$, and taking $M$ large enough, we obtain
$${\sum\limits_{k=1}^q  \sum\limits_{i=1}^r \Big| \Big<L_n(x_i) - L_n(x_{i-1}), \nabla g_k^{H, \beta}(L_n( x_{i-1})) \Big> \Big| \over  \|L_n(\sigma)-L_n(\tau)\|_1 } \leq 1-\delta/4.$$

\medskip
\noindent
Thus equation (\ref{eqn:meandist}) will imply 

{\small
\beas
\mathbb{E}_{\sigma, \tau}[d(X,Y)] & \leq & d(\sigma,\tau) \left[ 1-{1 \over n}\left(1-{\sum\limits_{k=1}^q  \sum\limits_{i=1}^r \Big| \Big<L_n(x_i) - L_n(x_{i-1}), \nabla g_k^{H, \beta}(L_n( x_{i-1})) \Big> \Big|+c \, \varepsilon^2 \over  \|L_n(\sigma)-L_n(\tau)\|_1 }\right) \right] \\
& \leq &  d(\sigma,\tau) \left[ 1 -{1 \over n}\big(1- (1-\delta/4)-\delta/20 \big) \right]\\
& = &  d(\sigma,\tau) \left[ 1 -{1 \over n}\delta/5 \right]
\eeas
}
as $~{c \, \varepsilon^2 \over  \|L_n(\sigma)-L_n(\tau)\|_1 } ~\leq ~c \, \ve ~\leq ~\delta/20~$ for 
$\ve$ small enough.\\

\bigskip
\noindent
{\bf Case II.} Suppose $L_n(\tau)=z$ and $L_n(\sigma)=w$, where $\|z-z_\beta\|_1 < \ve$ and $\|w-z_\beta\|_1< \ve'$.

\medskip
\noindent
Similarly to (\ref{eqn:grad}), equation (\ref{eqn:true}) implies for $n$ large enough,
\beas
\mathbb{E}[d(X,Y)] &  \leq & d(\sigma,\tau) \cdot \left[1-{1 \over n}\left(1-{\|g^{H,\beta}\big(L_n(\sigma)\big)-g^{H,\beta}\big(L_n(\tau)\big)\|_1 \over  \|L_n(\sigma)-L_n(\tau)\|_1} \right)  \right] + O\left( {1 \over n^2}\right)\\
&  \leq & d(\sigma,\tau) \cdot \left[1-{\gamma \over n} \right] + O\left( {1 \over n^2}\right)\\
&  \leq & d(\sigma,\tau) \cdot \left[1-{\gamma \over 2n} \right] 
\eeas
 by Condition \ref{ass:localbound} (see also discussion following Condition \ref{ass:localbound}).
\ep

We now state and prove the main theorem of the paper that yields sufficient conditions for rapid mixing of the Glauber dynamics of the class of statistical mechanical models discussed.

\begin{thm}
\label{thm:main}
Suppose $H(z)$ and $\beta>0$ are such that  Condition \ref{RS} and Condition \ref{ass:localbound} are satisfied. Then the mixing time of the Glauber dynamics satisfies
$$t_{\text{{\em \mix}}} = O(n \log n)$$
\end{thm}

\begin{proof}  
Let $\ve'>0$ and $\alpha>0$ be as in Lemma \ref{lemma:contract}.
Let $(X^t, Y^t)$ be a coupling of the Glauber dynamics such that $X^0 \overset{dist}{=} P_{n, \beta}$, the stationary distribution.  
Then, for sufficiently large $n$,
\beas
\lefteqn{ \| P^t(Y^0, \cdot) - P_{n, \beta} \|_{\TV} } \\
& \hsp \hsp \leq & P \{ X^t \neq Y^t \} \\
& \hsp \hsp = & P \{ d(X^t, Y^t ) \geq 1 \} \\
& \hsp \hsp \leq & \mathbb{E} [ d(X^t,Y^t)] \\
& \hsp \hsp = & \mathbb{E} [ {\scriptstyle \mathbb{E}[d(X^t,Y^t)~|X^{t-1},Y^{t-1}]} ]\\
& \hsp \hsp \leq & \mathbb{E} [ {\scriptstyle \mathbb{E}[d(X^t,Y^t)~|X^{t-1},Y^{t-1}]}~|~\|L_n(X^{t-1}) - z_\beta\|_1 < \varepsilon' ] \cdot P\{\|L_n(X^{t-1}) - z_\beta\|_1<\varepsilon' \} \\
& & + n P\{\|L_n(X^{t-1}) - z_\beta\|_1 \geq \varepsilon' \} .
\eeas

\noindent
By Lemma \ref{lemma:contract}, we have
\bea
\label{eqn:meandistcontract}
\lefteqn{ \mathbb{E} [ {\scriptstyle \mathbb{E}[d(X^t,Y^t)~|X^{t-1},Y^{t-1}]}~|~\|L_n(X^{t-1}) - z_\beta\|_1 < \varepsilon' ] } \nonumber \\
& \hsp \hsp \hsp & \leq e^{-\alpha /n} \mathbb{E} [ d(X^{t-1},Y^{t-1})~|~\|L_n(X^{t-1}) - z_\beta\|_1 < \varepsilon' ] 
\eea

\noi
By iterating (\ref{eqn:meandistcontract}), it follows that 
\beas
\lefteqn{ \| P^t(Y^0, \cdot) - P_{n, \beta, K} \|_{\mbox{{\small TV}}} } \\
& \hsp \hsp  \leq & e^{-\alpha /n} \mathbb{E} [ d(X^{t-1},Y^{t-1})~|~\|L_n(X^{t-1}) - z_\beta\|_1 < \varepsilon' ] \cdot P\{\|L_n(X_{t-1}) - z_\beta\|_1<\varepsilon' \} \\
& & + n P\{\|L_n(X^{t-1}) - z_\beta\|_1 \geq \varepsilon' \} \\
& \hsp \hsp  \leq & e^{-\alpha /n} \mathbb{E} [ d(X^{t-1},Y^{t-1})]+nP\{\|L_n(X^{t-1}) - z_\beta\|_1 \geq \varepsilon' \} \\
& \hsp \hsp \vdots & \vdots\\
& \hsp \hsp \leq & e^{-\alpha t/n} \mathbb{E} [ d(X^0,Y^0)] +n \sum_{s=0}^{t-1} P\{\|L_n(X^s) - z_\beta\|_1 \geq \varepsilon' \} \\
& \hsp \hsp = & e^{-\alpha t/n} \mathbb{E} [ d(X^0,Y^0)]+n t P_{n, \beta}\{\|L_n(X^0) - z_\beta\|_1 \geq \ve'\} \\
& \hsp \hsp \leq & n e^{-\alpha t/n}+ntP_{n, \beta}\{\|L_n(X^0) - z_\beta\|_1 \geq \varepsilon' \} .
\eeas

\noi
We recall the LDP limit (\ref{eqn:ldplimit}) for $\beta$ in the single phase region $B$, 
\[ P_{n, \beta}\{ L_n(X^0) \in dx \} \Longrightarrow \delta_{z_\beta} \hsp \mbox{as $n \goto \infty$}. \]
Moreover, for any $\gamma'>1$ and $n$ sufficiently large, by the LDP upper bound, we have
\beas
\|P^t(Y^0, \cdot) - P_{n, \beta} \|_{\TV} & \leq & n e^{-\alpha t/n}+n t P_{n, \beta}\{\|L_n(X^0) - z_\beta\|_1 \geq \varepsilon' \}  \\
&  <  & n e^{-\alpha t/n} +t n e^{-{n \over \gamma'}I_{\beta }(\varepsilon')} .
\eeas
For $t = {n \over \alpha} (\log n + \log (2/\ve'))$, the above right-hand side converges to $\ve'/2$ as $n \goto \infty$. 
\ep

\bigskip

\section{Aggregate Path Coupling applied to the Generalized Potts Model} \label{GCWP} 

\medskip

In this section, we illustrate the strength of our main result of section \ref{genspinmodels}, Theorem \ref{thm:main}, by applying it to the generalized Curie-Weiss-Potts model (GCWP), studied recently in \cite{JKRW}.  The classical Curie-Weiss-Potts (CWP) model, which is the mean-field version of the well known Potts model of statistical mechanics \cite{Wu} is a particular case of the GCWP model with $r=2$.  While the mixing times for the CWP model has been studied in \cite{CDLLPS, KO}, these are the first results for the mixing times of the GCWP model.

\medskip

Let $q$ be a fixed integer and define $\Lambda = \{ e^1, e^2, \ldots, e^q \}$, where $e^k$ are the $q$ standard basis vectors of $\R^q$.  A  {\it configuration} of the model has the form $\omega = (\omega_1, \omega_2, \ldots, \omega_n) \in \Lambda^n$.  We will consider a configuration on a graph with $n$ vertices and let $X_i(\omega) = \omega_i$ be the {\it spin} at vertex $i$.  The random variables $X_i$'s for $i=1, 2, \ldots, n$ are independent and identically distributed with common distribution $\rho$.  

For the generalized Curie-Weiss-Potts model, for $r \geq 2$, the interaction representation function, defined in general in (\ref{eqn:IntRepFunct}), has the form
\[ H^r(z) = - \frac{1}{r} \sum_{j=1}^q z_j^r \]
and the generalized Curie-Weiss-Potts model is defined as the Gibbs measure
\be 
\label{eqn:GCWPgibbs}
P_{n, \beta, r} (B) = \frac{1}{Z_n(\beta)} \int_B \exp \left\{  -\beta n \, H^r\left(L_n(\omega) \right) \right\} dP_n 
\ee
where $L_n(\omega)$ is the empirical measure defined in (\ref{eqn:empmeasure}).

\medskip

In \cite{JKRW}, the authors proved that there exists a phase transition critical value $\beta_c(q, r)$ such that in the parameter regime $(q, r) \in \{2\} \times [2,4]$, the GCWP model undergoes a continuous, second-order, phase transition and for $(q,r)$ in the complementary regime, the GCWP model undergoes a discontinuous, first-order, phase transition.  This is stated in the following theorem.

\begin{thm}[Generalized Ellis-Wang Theorem]
\label{thm:EW}
Assume that $q \geq 2$ and $r \geq 2$.  Then there exists a critical temperature $\beta_c(q,r) > 0$ such that in the weak limit
\[ \lim_{n \goto \infty} P_{n, \beta, r} (L_n \in \cdot) = \left\{ \begin{array}{ll} \delta_{1/q(1, \ldots, 1)} & \mbox{if} \ \beta < \beta_c(q,r) \vspace{.1in} \\   \frac{1}{q} \sum_{i=1}^q \delta_{u(\beta, q, r) e^i + (1-u(\beta, q, r))/q(1, \ldots, 1)} & \mbox{if} \ \beta > \beta_c(q,r) \end{array} \right. \]
where $u(\beta, q, r)$ is the largest solution to the so-called mean-field equation 
\[ u = \frac{1 - \exp( \Delta(u))}{1+(q-1) \exp (\Delta (u))} \]
with $\Delta (u) :=-{\beta \over q^{r-1}}\big[(1+(q-1)u)^{r-1}-(1-u)^{r-1} \big]$.  Moreover, for $(q, r) \in \{2\} \times [2,4]$, the function $\beta \mapsto u(\beta, q, r)$ is continuous whereas, in the complementary case, the function is discontinuous at $\beta_c(q,r)$.
\end{thm}

\medskip

For the GCWP model, the function $g_{\ell}^{H, \beta}(z)$ defined in general in (\ref{eqn:gell}) has the form
\[ 
g_{k}^{H, \beta}(z) = \left[ \partial_{k} \Gamma \right] (\beta \nabla H(z)) = \left[ \partial_{k} \Gamma \right] (\beta z) = {e^{\beta z_k^{r-1}} \over e^{\beta z_1^{r-1}}+ \ldots +e^{\beta z_q^{r-1}}}.
\]
For the remainder of this section, we will replace the notation $H, \beta$ and refer to $g^{H, \beta}(z)=\big(g_1^{H, \beta}(z),\hdots,g_q^{H, \beta}(z)\big)$ as simply $g^r(z)=\big(g_1^r(z),\hdots,g_q^r(z)\big)$.  As we will prove next, the rapid mixing region for the GCWP model is defined by the following value.

\bea
\label{eqn:mixcrit} 
\beta_s(q, r) := \sup \left\{ \beta \geq 0 : g_k^r(z) < z_k  \ \mbox{for all $z \in \mathcal{P}$ such that}  \ z_k \in (1/q, 1] \right\}. 
\eea

\medskip

\begin{lemma} 
\label{lemma:CritIneq}
If $\beta_c(q,r)$ is the critical value derived in \cite{JKRW} and defined in Theorem \ref{thm:EW}, then
$$\beta_s(q,r) \leq \beta_c(q,r)$$ 
\end{lemma}
\begin{proof}
We will prove this lemma by contradiction. Suppose $\beta_c(q,r) < \beta_s(q,r)$. Then there exists $\beta$ such that 
$$\beta_c(q,r) < \beta < \beta_s(q,r).$$
Then, by Theorem \ref{thm:EW}, since $\beta_c(q,r) <\beta$, there exists $u>0$ satisfying the following inequality
\begin{equation}\label{ineq:du}
u < {1-e^{\Delta(u)} \over 1+(q-1)e^{\Delta(u)}},
\end{equation}
where $~\Delta(u):=-{\beta \over q^{r-1}}\big[(1+(q-1)u)^{r-1}-(1-u)^{r-1} \big]$. Here, the above inequality (\ref{ineq:du}) rewrites as 
\begin{equation}\label{ineq:du2}
e^{\Delta(u)}=\exp\left\{\beta\left[\left({1-u \over q} \right)^{r-1}-\left({1+(q-1)u \over q} \right)^{r-1}\right]\right\} ~<~ {1-u \over (q-1)u+1}.
\end{equation}
Next, we substitute $\lambda=(1-u){q-1 \over q}$ into the above inequality (\ref{ineq:du2}), obtaining
\begin{equation}\label{ineq:du3}
\exp\left\{\beta\left[\left({\lambda \over q-1} \right)^{r-1}-\left(1-\lambda \right)^{r-1}\right]\right\} ~<~ {\lambda \over (1-\lambda)(q-1)}.
\end{equation}
Now, consider 
$$z=\left(1-\lambda, {\lambda \over q-1},\hdots, {\lambda \over q-1} \right).$$
Observe that $~z_1=1-\lambda=1-(1-u){q-1 \over q}={1+u(q-1) \over q}>{1 \over q}~$ as $u>0$.
Here, the inequality (\ref{ineq:du3}) can be consequently rewritten in terms of the above selected $z$ as follows
$$z_1=1-\lambda ~<~{e^{\beta(1-\lambda)^{r-1}} \over e^{\beta(1-\lambda)^{r-1}}+(q-1)e^{\beta\big({\lambda \over q-1}\big)^{r-1}}}=g_1^r(z),$$
thus contradicting $\beta < \beta_s(q,r)$. Hence $~\beta_s(q,r) \leq \beta_c(q,r)$.
\end{proof}

Combining Theorem \ref{thm:EW} and Lemma \ref{lemma:CritIneq} yields that for parameter values $(q, r)$ in the continuous, second-order phase transition region $\beta_s(q,r) = \beta_c(q,r)$, whereas in the discontinuous, first-order, phase transition region, $\beta_s(q,r)$ is strictly less than $\beta_c(q,r)$. This relationship between the equilibrium transition critical value and the mixing time transition critical value was also proved for the mean-field Blume-Capel model discussed in section \ref{MFBC}.  This appears to be a general distinguishing feature between models that exhibit the two distinct type of phase transition.  We now prove rapid mixing for the generalized Curie-Weiss-Potts model for $\beta < \beta_s(q,r)$ using the aggregate path coupling method derived in section \ref{genspinmodels}.

\medskip
We state the lemmas that we prove below, and the main result for the Glauber dynamics of the generalized Curie-Weiss-Potts model, a Corollary  to Theorem \ref{thm:main}. Let $\beta_s(q)$ be the mixing time critical value for the GCWP model defined in (\ref{eqn:mixcrit}).  

\begin{lemma}\label{cwp1}
Condition \ref{uniform} and Condition \ref{RS} are satisfied for all $\beta < \beta_s(q)$.
\end{lemma}

\begin{lemma}\label{cwp2}
Condition \ref{ass:localbound} is satisfied for all $\beta < \beta_s(q)$.
\end{lemma}

\begin{cor}
\label{cor:CWPrapid} 
If $\beta < \beta_s(q)$, then
\[ t_{\mbox{{\em \mix}}} = O(n \log n). \]
\end{cor}
\bp
Condition \ref{RS} and Condition \ref{ass:localbound} required for Theorem \ref{thm:main} are satisfied by Lemma \ref{cwp1} and Lemma \ref{cwp2}.
\ep

\medskip
\noindent
\bp[Proof of Lemma \ref{cwp2}]
Denote $~z'=(z'_1,\hdots,z'_q)=z-z_\beta$. Then by Taylor's Theorem, we have
\bea \label{zero} 
\limsup_{z \rightarrow z_\beta} {\|g^r(z)-g^r(z_\beta)\|_1 \over \|z-z_\beta\|_1} & = & \limsup_{z \rightarrow z_\beta} {\sum\limits_{k=1}^q \left|{e^{\beta z_k^{r-1}} \over \sum\limits_{j=1}^q e^{\beta z_j^{r-1}}} -{1 \over q} \right| \over \sum\limits_{k=1}^q \left|z_k -{1 \over q} \right|} \nonumber\\
& = & \lim_{z' \rightarrow 0} \frac{ \sum\limits_{k=1}^q \left|{\beta (r-1) \left({1 \over q} \right)^{r-2} z'_k +O\big((z'_1)^2+\hdots+(z'_q)^2\big) \over q+O\big((z'_1)^2+\hdots+(z'_q)^2\big)  }  \right| }{ \sum\limits_{k=1}^q \left|z'_k  \right| } \nonumber\\
&= & \frac{\beta (r-1) \left({1 \over q} \right)^{r-2}}{q}.
\eea
In \cite{CDLLPS}, it was shown that for $r=2$, $~\beta < \beta_s(q)< \beta_c(q) < q$.  Therefore, the last expression above is less than $1$ and we conclude that 
\[ \limsup_{z \rightarrow z_\beta} {\|g^r(z)-g^r(z_\beta)\|_1 \over \|z-z_\beta\|_1} < 1. \qedhere \]
\ep

\medskip
\noindent
{\it Proof of Lemma \ref{cwp1}.} First, we prove that the family of straight lines connecting to the equilibrium point $z_\beta=\left(1/q,\hdots,1/q\right)$ is a neo-geodesic family as it was defined following Condition \ref{uniform}.  Specifically, for any $z = (z_1, z_2, \ldots, z_q) \in \mathcal{P}$ define the line path $\rho$ connecting $z$ to $z_\beta$ by 
\be
\label{eqn:linepath2} 
z(t) = {1 \over q} (1-t) + z \, t, \hsp 0 \leq t \leq 1 
\ee

\noindent
Then, along this straight-line path $\rho$, the aggregate $g$-variation has the form
\[ D_\rho^g (z,z_\beta) := \sum\limits_{k=1}^q\int\limits_\rho \Big| \Big<\nabla g_k^r (y), dy \Big> \Big| =  \sum\limits_{k=1}^q \int_0^1 \left| \frac{d}{dt} [g_k^r (z(t))] \right| \, dt \]

\medskip
\noindent
Next, for all $k = 1, 2, \ldots, q$ and $t \in [0,1]$, denote 
\[ z(t)_k = {1 \over q} (1-t) + z_k t \]
Then
\be
\label{eqn:coordg2} 
g_k^r(z(t)) = \frac{e^{\beta \big((1/q)(1-t) + z_k t \big)^{r-1}}}{\sum_{j=1}^q e^{\beta \big((1/q) (1-t) + z_j t \big)^{r-1}}}  
\ee
and 
\be
\label{eqn:dg}
\frac{d}{dt} \big[g_k^r(z(t)) \big] = \beta (r-1) g_k^r (z(t)) \Big[ \left(\frac{1}{q} (1-t) + z_k t \right)^{r-2} \left(z_k- \frac{1}{q} \right) - \langle z-z_\beta, g^r(z(t)) \rangle_\rho \Big] 
\ee
where $\langle z-z_\beta, g^r(z(t)) \rangle_\rho$ is the weighted inner product
\[ \langle z-z_\beta, g^r(z(t)) \rangle_\rho := \sum_{j=1}^q g_j^r(z(t)) \left(z_k- \frac{1}{q} \right)  \left(\frac{1}{q} (1-t) + z_k t \right)^{r-2} \]

\medskip
\noindent
Now, observe that for $z(t)$ as in (\ref{eqn:linepath2}) with $z \not= z_\beta$, the inner product $\langle (z-z_\beta), g^r(z(t)) \rangle_\rho$ is monotonically increasing in $t$ since
\[  \frac{d}{dt} \langle z-z_\beta, g^r(z(t)) \rangle_\rho \geq \beta (r-1) \, \mbox{Var}_{g^r}\left( \left(z_k-\frac{1}{q} \right) \left(\frac{1}{q} (1-t) + z_j t  \right)^{r-1} \right) > 0 \]
where $\mbox{Var}_{g^r}(\cdot)$ is the variance with respect to $g^r$.

So $\langle z-z_\beta, g^r(z(t)) \rangle_\rho$ begins at $\langle z-z_\beta, g^r(z(0)) \rangle_\rho =\langle z-z_\beta, z_\beta \rangle =0$ and increases for all $t \in (0,1)$.

\bigskip
\noindent
The above monotonicity yields the following claim about the behavior of $g_k^r(z(t))$ along the straight-line path $\rho$.

\begin{itemize}
\item[(a)] If $z_k \leq 1/q$, then $g_k^r(z(t))$ is monotonically decreasing in $t$.

\item[(b)] If $z_k > 1/q$, then $g_k^r(z(t))$ has at most one critical point $t_k^\ast$ on $(0,1)$.
\end{itemize}

\noindent
The above claim (a) follows immediately from (\ref{eqn:dg}) as $\langle z-z_\beta, g^r(z(t)) \rangle_\rho >0$ for $t>0$. Claim (b) also follows from (\ref{eqn:dg}) as its right-hand side, $z_k-1/q>0$ and $\langle z-z_\beta, g^r(z(t)) \rangle_\rho$ is increasing. Thus there is at most one point $t_k^\ast$ on $(0,1)$ such that $~\frac{d}{dt} \big[g_k^r(z(t)) \big] =0$.

\bigskip
\noi
Next, define 
\[ A_z = \{ k : z_k > 1/q \} \]
Then the aggregate $g$-variation can be split into 
\[ D_\rho^g (z,z_\beta) = \sum_{k \in A_z} \int_0^1 \left| \frac{d}{dt} [g_k^r(z(t))] \right| \, dt + \sum_{k \notin A_z} \int_0^1 \left| \frac{d}{dt} [g_k^r(z(t))] \right| \, dt \]
For $k \notin A_z$, claims (a) and (b) imply
\[ \int_0^1 \left| \frac{d}{dt} [g_k^r(z(t))] \right| \, dt = - \int_0^1 \frac{d}{dt} [g_k^r(z(t))] \, dt = g_k^r(z(0)) - g_k^r(z(1)) = \frac{1}{q} - g_k^r(z) \]
For $k \in A_z$, let $t_k = \max\{ t_k^\ast, 1\}$ ,where $t_k^\ast$ is defined in (b).  Then, we have 
\[ \int_0^1 \left| \frac{d}{dt} [g_k^r(z(t))] \right| \, dt = \int_0^{t_k^\ast} \frac{d}{dt} [g_k^r(z(t))] \, dt - \int_{t_k^\ast}^1 \frac{d}{dt} [g_k^r(z(t))] \, dt = 2 g_k^r(z(t_k^\ast)) - g_k^r(z) - \frac{1}{q} \]
Combining the previous two displays, we get
\beas
D_\rho^g (z,z_\beta) & = & \sum_{k \in A} \left( 2 g_k^r(z(t_k^\ast)) - g_k^r(z) - \frac{1}{q} \right) + \sum_{k \notin A} \left(\frac{1}{q} - g_k^r(z)  \right) \\
& = & 2 \sum_{k \in A} \left( g_k^r(z(t_k^\ast)) - \frac{1}{q} \right)
\eeas
Since $\beta < \beta_s$ and $k \in A_z$, we have 
\[ g_k^r(z(t_k^\ast)) < z(t_k^\ast)_k \leq z(1)_k = z_k \]
and we conclude that 
\[ D_\rho^g (z,z_\beta) < 2 \sum_{k \in A} \left( z_k - \frac{1}{q} \right) = \| z - z_\beta \|_1 \]
Thus
$${d_g(z,z_\beta) \over \| z - z_\beta \|_1} \leq {D_\rho^g (z,z_\beta) \over \| z - z_\beta \|_1} <1 \quad \text { for all } z \not=z_\beta \text{ in } \mathcal{P}.$$

\bigskip
\noi
Next, since we are dealing with the straight line segments $\rho$,
$$\limsup_{z \rightarrow z_\beta}{D_\rho^g (z,z_\beta) \over \| z - z_\beta \|_1}=\limsup_{z \rightarrow z_\beta} {\|g(z)-g(z_\beta)\|_1 \over \|z-z_\beta\|_1} <1$$
by (\ref{zero}), the Mean Value Theorem, and $H(z) \in \mathcal{C}^3$. This, in turn, guarantees the continuity required for Condition \ref{uniform}:
$$\limsup_{z \rightarrow z_\beta}{d_g(z,z_\beta) \over \| z - z_\beta \|_1} \leq \limsup_{z \rightarrow z_\beta}{D_\rho^g (z,z_\beta) \over \| z - z_\beta \|_1}<1$$
Thus Condition \ref{uniform} is proved for the CWP model. Moreover this proves that the family of straight line segments $\rho$ is a neo-geodesic family (see definition following Condition \ref{uniform}). Indeed, there is $\delta \in (0,1)$ such that 
$$\left\{\rho:~z(t)={1 \over q}(1-t)+zt, ~~z \in \mathcal{P} \right\} \quad \text{ is a } {\bf NG}_\delta \text{ family of smooth curves,}$$
 i.e. $\forall z\not= z_\beta$ in $\mathcal{P}$, and corresponding $\rho:~z(t)={1 \over q}(1-t)+zt$,
$$ {D_\rho^g (z,z_\beta) \over \| z - z_\beta \|_1} \leq 1-\delta/2 \qquad$$

\bigskip
\noi
Since the family of straight line segments $\rho$ is a neo-geodesic family ${\bf NG}_\delta$, the integrals
$$ D_\rho^g (x, z) := \sum\limits_{k=1}^q\int\limits_{\rho} \Big| \Big<\nabla g_k^r(y), dy \Big> \Big| $$
can be uniformly approximated by the corresponding  Riemann sums of small enough step size by the Mean Value Theorem as $H(z) \in \mathcal{C}^3$ and therefore each $g_k^r(z) \in \mathcal{C}^2$.
That is, there exists a constant $C>0$ that depends on the second partial derivatives of $g^r(z)=\big(g_1^r(z),\hdots,g_q^r(z)\big)$, such that for $\varepsilon>0$ small enough, the curve  $\rho={1 \over q}(1-t)+zt$ in the family  ${\bf NG}_\delta$ that connects $z_\beta$ to $z$ satisfies
$$\left| \sum_{k=1}^q  \sum_{i=1}^r \Big| \Big<z_i - z_{i-1}, \nabla g_k^r(z_{i-1}) \Big> \Big|-D_\rho^g (z,z_\beta) \right| < C r \ve^2
\qquad \forall z \in \mathcal{P} \text{ s.t. } \|z-z_\beta\|_1 \geq \varepsilon $$
for a sequence of points $z_0=z_\beta,z_1,\hdots,z_r=z \in \mathcal{P}$ interpolating $\rho$ such that
$$\varepsilon \leq \|z_i - z_{i-1}\|_1  < 2\varepsilon \quad \text{ for } i=1,2,\hdots, r.$$
Hence
$${\sum\limits_{k=1}^q  \sum\limits_{i=1}^r \Big| \Big<z_i - z_{i-1}, \nabla g_k^r(z_{i-1}) \Big> \Big| \over \|z-z_\beta\|_1} \leq 1-\delta/2+C\ve \leq 1-\delta/3$$
for $\ve \leq \delta/(6C)$. This concludes the proof of Condition \ref{RS}. \hfill $\square$

\newpage



\bibliographystyle{amsplain}

\end{document}